\documentclass{article}

\usepackage{latexsym}
\usepackage{amsmath, amsfonts, amsthm}
\usepackage{epsfig}
\usepackage{stmaryrd}
\usepackage{multicol}
\usepackage{pdfsync}
\usepackage{dsfont}
\usepackage{comment}
\usepackage{algorithmic}
\usepackage{algorithm}
\usepackage{caption}
\usepackage{psfrag}
\usepackage{xcolor}
\usepackage{graphics}
\newcommand{\la}{\left \langle}
\newcommand{\ra}{\right\rangle}

\newcommand{\norm}[1]{\left\lVert #1 \right\rVert}

\newtheorem{theorem}{Theorem}[section]

\newtheorem{lemma}[theorem]{Lemma}

\theoremstyle{definition}

\theoremstyle{remark}
\newtheorem{remark}[theorem]{Remark}

\newtheorem{condition}[theorem]{Condition}
\numberwithin{equation}{section}

\usepackage{setspace}
\title{Stochastic Gradient Descent in Continuous Time}
\author{Justin Sirignano\footnote{Department of Industrial and Enterprise Systems Engineering, University of Illinois at Urbana-Champaign, Urbana, Email: jasirign@illinois.edu} \phantom{.}  and Konstantinos Spiliopoulos\footnote{Department of Mathematics and Statistics, Boston University, Boston, E-mail: kspiliop@math.bu.edu} \thanks{Research of K.S. supported in part by the National Science Foundation (DMS 1550918)}  \thanks{Computations for this paper were supported by a Blue Waters supercomputer grant.} \\
}

\date{\today}

\begin{document}

\maketitle
\begin{abstract}
Stochastic gradient descent in continuous time (SGDCT) provides a computationally efficient method for the statistical learning of continuous-time models, which are widely used in science, engineering, and finance.  %\textcolor{blue}{In particular, this paper focuses on applications in finance, such as model estimation for stocks, bonds, interest rates, and financial derivatives.  SGDCT can also be used for the optimization of %high-dimensional continuous-time models, such as American options.}
The SGDCT algorithm follows a (noisy) descent direction along a continuous stream of data.  SGDCT performs an online parameter update in continuous time, with the parameter updates $\theta_t$ satisfying a stochastic differential equation.  We prove that $\lim_{t \rightarrow \infty} \nabla \bar g(\theta_t) = 0$ where  $\bar g$ is a natural objective function for the estimation of the continuous-time dynamics.  The convergence proof leverages ergodicity by using an appropriate Poisson equation to help describe the evolution of the parameters for large times.  For certain continuous-time problems, SGDCT has some promising advantages compared to a traditional stochastic gradient descent algorithm.  This paper mainly focuses on applications in finance, such as model estimation for stocks, bonds, interest rates, and financial derivatives.  SGDCT can also be used for the optimization of high-dimensional continuous-time models, such as American options. As an example application, SGDCT is combined with a deep neural network to price high-dimensional American options (up to 100 dimensions).
\end{abstract}

\section{Introduction}

This paper develops a statistical learning algorithm for continuous-time models, which are common in science, engineering, and finance.  We study its theoretical convergence properties as well as its computational performance in a number of benchmark problems. Although the method is broadly applicable, this paper mainly focuses on applications in finance. Given a continuous stream of data, stochastic gradient descent in continuous time (SGDCT) can estimate unknown parameters or functions in stochastic differential equation (SDE) models for stocks, bonds, interest rates, and financial derivatives.  The statistical learning algorithm can also be used for the optimization of high-dimensional continuous-time models, such as American options.  High-dimensional American options have been a longstanding computational challenge in finance.  SGDCT is able to accurately solve American options even in 100 dimensions.

Batch optimization for the statistical estimation of continuous-time models can be impractical for large datasets where observations occur over a long period of time.  Batch optimization takes a sequence of descent steps for the model error for the entire observed data path.  Since each descent step is for the model error for the \emph{entire observed data path}, batch optimization is slow (sometimes impractically slow) for long periods of time or models which are computationally costly to evaluate (e.g., partial differential equations).  Typical existing approaches in the financial statistics literature use batch optimization.

SGDCT provides a computationally efficient method for statistical learning over long time periods and for complex models.  SGDCT \emph{continuously} follows a (noisy) descent direction \emph{along the path of the observation}; this results in much more rapid convergence.  Parameters are updated online in continuous time, with the parameter updates $\theta_t$ satisfying a stochastic differential equation.  We prove that $\lim_{t \rightarrow \infty} \nabla \bar g(\theta_t) = 0$ where  $\bar g$ is a natural objective function for the estimation of the continuous-time dynamics.

%Applications of SGDCT include model estimation and continuous-time optimization.

Consider a diffusion $X_t \in \mathcal{X}=\mathbb{R}^{m}$:
\begin{eqnarray}
d X_t = f^{\ast}(X_t) dt + \sigma d W_t.
\label{ClassofEqns}
\end{eqnarray}
The goal is to statistically estimate a model $f(x, \theta)$ for $f^{\ast}(x)$ where $\theta \in \mathbb{R}^n$.  The function $f^{\ast}(x)$ is unknown.  $W_t \in \mathbb{R}^m$ is a standard Brownian motion.  The diffusion term $W_t$ represents any random behavior of the system or environment.  The functions $f(x, \theta)$ and $f^{\ast}(x)$ may be non-convex.

The stochastic gradient descent update in continuous time follows the SDE:
\begin{eqnarray}
d \theta_t =  \alpha_{t}\big [  \nabla_{\theta}  f(X_t, \theta_t)(\sigma \sigma^{\top})^{-1} d X_t -    \nabla_{\theta}    f(X_t, \theta_t)(\sigma \sigma^{\top})^{-1} f(X_t, \theta_t)  dt\big],
\label{SDEMain}
\end{eqnarray}
where $\nabla_{\theta}  f(X_t; \theta_t)$ is matrix valued and $\alpha_{t}$ is the learning rate.  The parameter update (\ref{SDEMain}) can be used for both statistical estimation given previously observed data as well as online learning (i.e., statistical estimation in real-time as data becomes available).  SGDCT will still converge if $\sigma \sigma^{\top}$ in (\ref{SDEMain}) is replaced by the identity matrix $I$.

Using the proposed approach of this paper, the stochastic gradient descent algorithm (\ref{SDEMain}) can also be generalized to the case where $\sigma$ is a variable coefficient $\sigma^{\ast}(X_t)$.  In that case, a model $\sigma(x, \nu)$ is also learned for $\sigma^{\ast}(x)$ where $\nu \in \mathbb{R}^{k}$ is an additional set of parameters.  (To be more precise, $\sigma(x, \nu)\sigma^{\top}(x, \nu)$ is learned for $\sigma^{\ast}(x) \sigma^{\ast, \top} (x)$ since $\sigma^{\ast}(x)$ is not identifiable.)  See Section \ref{DiffusionCoefficientFunction} for details and the corresponding convergence proof.

We assume that $X_{t}$ is sufficiently ergodic (to be concretely specified later in the paper) and that it has some well-behaved $\pi(dx)$ as its unique invariant measure. As a general notation, if $h(x,\theta)$ is a generic $L^{1}(\pi)$ function, then we define its average over $\pi(dx)$ to be
\[
\bar{h}(\theta)=\int_{\mathcal{X}} h(x,\theta)\pi(dx).
\]

Let us set
\[
g(x,\theta)=\frac{1}{2}\left\|f(x,\theta)-f^{\ast}(x)\right\|^{2}_{\sigma\sigma^{\top}}=\frac{1}{2}\left<f(x,\theta)-f^{\ast}(x),\left(\sigma\sigma^{\top}\right)^{-1}(f(x,\theta)-f^{\ast}(x))\right>.
\]

The gradient $\nabla_{\theta} g(X_t, \theta)$ cannot be evaluated since $f^{\ast}(x)$ is unknown.  However, $d X_t = f^{\ast}(X_t) dt + \sigma d W_t$ is a noisy estimate of $f^{\ast}(x) dt$, which leads to the algorithm (\ref{SDEMain}).  SGDCT follows a noisy descent direction along a continuous stream of data produced by $X_t$.

Heuristically, it is expected that $\theta_t$ will tend towards the minimum of the function $\bar{g}(\theta)=\int_{\mathcal{X}} g(x,\theta)\pi(dx)$.  The data $X_t$ will be correlated over time, which complicates the mathematical analysis.  This differs from the standard discrete-time version of stochastic gradient descent where the the data is usually considered to be  i.i.d. at every step.

\subsection{Literature Review}

In this paper we show that if $\alpha_{t}$ is appropriately chosen then $\nabla \bar{g}(\theta_{t})\rightarrow 0$ as $t\rightarrow\infty$ with probability 1 (see Theorem \ref{T:MainTheorem}).  Results like this have been previously derived for stochastic gradient descent in discrete time; see \cite{BertsekasThitsiklis2000} and \cite{Benveniste}.  \cite{BertsekasThitsiklis2000} proves convergence in the absence of the $X$ term.  \cite{Benveniste} proves convergence of stochastic gradient descent in discrete time with the $X$ process but requires stronger conditions than \cite{BertsekasThitsiklis2000}.

Although stochastic gradient descent for discrete time has been extensively studied, stochastic gradient descent in continuous time has received relatively little attention.  We refer readers to \cite{Benveniste,KushnerYin} and \cite{BertsekasThitsiklis2000} for a thorough review of the very large literature on stochastic gradient descent.  There are also many algorithms which modify traditional stochastic gradient descent (stochastic gradient descent with momentum, Adagrad, RMSprop, etc.).  For a review of these variants of stochastic gradient descent, see \cite{Goodfellow}.  We mention below the prior work which is most relevant to our paper.

Our approach and assumptions required for convergence are most similar to \cite{BertsekasThitsiklis2000}, who prove convergence of discrete-time stochastic gradient descent in the absence of the $X$ process.  The presence of the $X$ process is essential for considering a wide range of problems in continuous time, and showing convergence with its presence is considerably more difficult.  The $X$ term introduces correlation across times, and this correlation does not disappear as time tends to infinity.  This makes it challenging to prove convergence in the continuous-time case.  In order to prove convergence, we use an appropriate Poisson equation associated with $X$ to describe the evolution of the parameters for large times.

\cite{NemirovskiSGD} proves, in a setting different than ours, convergence in $L^2$ of projected stochastic gradient descent in discrete time for convex functions. In projected gradient descent, the parameters are projected back into an a priori chosen compact set.  Therefore, the algorithm cannot hope to reach the minimum if the minimum is located outside of the chosen compact set.  Of course, the compact set can be chosen to be very large for practical purposes.  Our paper considers unconstrained stochastic gradient descent in continuous time and proves the almost sure convergence $\nabla \bar{g}(\theta_{t})\rightarrow 0$ as $t\rightarrow\infty$ taking into account the $X$ component as well.  We do not assume any stability conditions on $X$ (except that it is ergodic with a unique invariant measure).

Another approach for proving convergence of discrete-time stochastic gradient descent is to show that the algorithm converges to the solution of an ODE which itself converges to a limiting point.  This is the approach of \cite{Benveniste}.  See also \cite{KushnerYin}.  This method, sometimes called the ``ODE method", requires the assumption that the iterates (i.e., the model parameters which are being learned) remain in a bounded set with probability one.  It is unclear whether the ODE method of proof can be successfully used to show convergence for a continuous-time stochastic gradient descent scheme.  In this paper we follow a potentially more straightforward method of proof by analyzing the speed of convergence to equilibrium with an appropriately chosen Poisson type of equation.

\cite{Raginsky} studies continuous-time stochastic mirror descent in a setting different than ours.  In the framework of \cite{Raginsky}, the objective function is known.  In this paper, we consider the statistical estimation of the unknown dynamics of a random process (i.e. the $X$ process satisfying (\ref{ClassofEqns})).

Statisticians and financial engineers have actively studied parameter estimation of SDEs, although typically not with statistical learning or machine learning approaches.  The likelihood function will usually be calculated from the \emph{entire observed path of $X$} (i.e., batch optimization) and then maximized to find the maximum likelihood estimator (MLE).  Unlike in this paper, the actual optimization procedure to maximize the likelihood function is often not analyzed.

Some relevant publications in the financial statistics literature include \cite{Sahalia1}, \cite{Sahalia2}, \cite{Basawa}, and \cite{Elerian}.  \cite{Basawa} derives the likelihood function for continuously observed $X$.  The MLE can be calculated via batch optimization.  \cite{Sahalia1} and \cite{Sahalia2} consider the case where $X$ is discretely observed and calculate MLEs via a batch optimization approach.   \cite{Elerian} estimates parameters by a Bayesian approach.  %Readers are referred to \cite{Nielsen} for a comprehensive overview of the financial statistics field.
Readers are referred to \cite{bishwal2008parameter,KutoyantsStatisticalInference,rao1999statistical} for thorough reviews of classical statistical inference methods for stochastic differential equations.

%We review some of the financial statistics literature and refer readers to \cite{Nielsen} for a comprehensive overview.  For continuously observed $X$, the likelihood function is known and batch optimization can be used to estimate parameters (for example, see \cite{Basawa}).

%Previous approaches in the statistical finance literature have used batch optimization, while this paper uses a stochastic gradient descent method.

\subsection{Applications of SGDCT}
Continuous-time models are especially common in finance.  Given a continuous stream of data, the stochastic gradient descent algorithm can be used to estimate unknown parameters or functions in SDE models for stocks, bonds, interest rates, and financial derivatives.  Numerical analysis of SGDCT for two common financial models is included in Sections \ref{OU}, \ref{OU2}, and \ref{CIR}.  The first is the well-known Ornstein-Uhlenbeck (OU) process (for examples in finance, see \cite{TimLeung}, \cite{TimLeung2}, \cite{Vasicek}, and \cite{Linetsky}).  The second is the multidimensional CIR process which is a common model for interest rates (for examples in finance, see \cite{Alfonsi}, \cite{Maghsoodi}, \cite{Brown}, \cite{Ahlip}, and \cite{Dai}).

Scientific and engineering models are also typically in continuous-time.  There are often coefficients or functions in these models which are uncertain or unknown; stochastic gradient descent can be used to learn these model parameters from data.  In Section \ref{NumericalAnalysis1}, we study the numerical performance for two example applications: Burger's equation and the classic reinforcement learning problem of balancing a pole on a moving cart.  Burger's equation is a widely used nonlinear partial differential equation which is important to fluid mechanics, acoustics, and aerodynamics.

A natural question is why use SGDCT versus a straightforward approach which (1) discretizes the continuous-time dynamics and then (2) applies traditional stochastic gradient descent.  For some of the same reasons that scientific models have been largely developed in continuous time, it can be advantageous to develop \emph{continuous-time} statistical learning for \emph{continuous-time models}.

SGDCT allows for the application of numerical schemes of choice to the theoretically correct statistical learning equation for continuous-time models.   This can lead to more accurate and more computationally efficient parameter updates.  Numerical schemes are always applied to continuous-time dynamics and different numerical schemes may have different properties for different continuous-time models.  A priori performing a discretization to the system dynamics and then applying a traditional discrete-time stochastic gradient descent scheme can result in a loss of accuracy.  For example, there is no guarantee that (1) using a higher-order accurate scheme to discretize the system dynamics and then (2) applying traditional stochastic gradient descent will produce a statistical learning scheme which is higher-order accurate in time.  Hence, it makes sense  to first develop the continuous-time statistical learning equation, and then apply the higher-order accurate numerical scheme.

Besides model estimation, SGDCT can be used to solve continuous-time optimization problems, such as American options.  We combine SGDCT with a deep neural network to solve American options in up to $100$ dimensions (see Section \ref{NumericalAnalysis2}).  An alternative approach would be to discretize the dynamics and then use the Q-learning algorithm (traditional stochastic gradient descent applied to an approximation of the discrete HJB equation).  However, Q-learning is biased while SGDCT is unbiased.  Furthermore, in SDE models with Brownian motions, the Q-learning algorithm can blow up as the time step size $\Delta$ becomes small; see Section \ref{NumericalAnalysis2} for details.

The convergence issue with Q-learning highlights the importance of studying continuous-time algorithms for continuous-time models.  It is of interest to show that (1) a discrete-time scheme converges to an appropriate continuous-time scheme as $\Delta \rightarrow 0$ and (2) the continuous-time scheme converges to the correct estimate as $t \rightarrow \infty$.  These are important questions since any discrete scheme for a continuous-time model incurs some error proportional to $\Delta$, and therefore $\Delta$ must be decreased to reduce error.  It is also important to note that in some cases, such as Q-learning, computationally expensive terms in the discrete algorithm (such as expectations over high-dimensional spaces) may become much simpler expressions in the continuous-time scheme (differential operators).

\subsection{Organization of Paper}

The paper is organized into five main sections.  Section \ref{Assumptions} presents the assumption and the main theorem. In Section \ref{S:ProofMainTheorem} we prove the main result of this paper for the convergence of continuous-time stochastic gradient descent.  The extension of the stochastic gradient descent algorithm to the case of a variable diffusion coefficient function is described in Section \ref{DiffusionCoefficientFunction}.  Section \ref{NumericalAnalysis1} provides numerical analysis of SGDCT for model estimation in several applications.  Section \ref{NumericalAnalysis2} discusses SGDCT for solving continuous-time optimization problems, particularly focusing on American options.

\section{Assumptions and Main Result} \label{Assumptions}
Before presenting the main result of this paper, Theorem \ref{T:MainTheorem}, let us elaborate on the standing assumptions. In regards to the learning rate $\alpha_{t}$ the standing assumption is
\begin{condition}\label{A:ExtraAssumption}
Assume that $\int_{0}^{\infty}\alpha_{t}dt=\infty$, $\int_{0}^{\infty}\alpha^{2}_{t}dt<\infty$, $\int_{0}^{\infty}|\alpha_{s}^{\prime}|ds<\infty$  and that there is a $p>0$ such that $\lim_{t\rightarrow\infty}\alpha_{t}^{2}t^{1/2+2p}$=0.
\end{condition}

A standard choice for $\alpha_t$ that satisfies Condition \ref{A:ExtraAssumption} is $\alpha_t = \frac{1}{C + t}$ for some constant $0<C<\infty$.  Notice that the condition $\int_{0}^{\infty}|\alpha_{s}^{\prime}|ds<\infty$ follows immediately from the other two restrictions for the learning rate if it is chosen to be a monotonic function of $t$.

Let us next discuss the assumptions that we impose on $\sigma$, $f^{\ast}(x)$ and $f(x,\theta)$. Condition \ref{A:LyapunovCondition} guarantees uniqueness and   existence of an invariant measure for the $X$ process. %We remark here that for the purposes of this paper one may just assume that $X_{t}$ has a unique invariant measure. However, we do state the  concrete sufficient Condition \ref{A:LyapunovCondition} in order for the user to have a reasonably general checkable condition.
\begin{condition}\label{A:LyapunovCondition}
We assume that $\sigma\sigma^{\top}$ is non-degenerate bounded diffusion matrix and $\lim_{|x|\rightarrow\infty}f^{\ast}(x)\cdot x=-\infty$
\end{condition}

In addition, with respect to $\nabla_{\theta}f(x,\theta)$ we assume that $\theta\in\mathbb{R}^{n}$ and we impose the following condition
\begin{condition}
\label{A:Assumption1}
\begin{enumerate}
\item  We assume that $\nabla_{\theta}g(x,\cdot)\in C^{2}(\mathbb{R}^{n})$ for all $x\in\mathcal{X}$, $\frac{\partial^{2}\nabla_{\theta}g}{\partial x^{2}}\in C\left(\mathcal{X},\mathbb{R}^{n}\right)$, $\nabla_{\theta}g(\cdot,\theta)\in C^{\alpha}\left(\mathcal{X}\right)$ uniformly in $\theta\in\mathbb{R}^{n}$ for some $\alpha\in(0,1)$ and that there exist $K$ and $q$ such that
    \[
    \sum_{i=0}^{2}\left|\frac{\partial^{i} \nabla_{\theta}g}{\partial \theta^{i}}(x,\theta)\right|\leq K\left(1+|x|^{q}\right).
    \]
\item For every $N>0$ there exists a constant $C(N)$ such that for all $\theta_{1},\theta_{2}\in\mathbb{R}^{n}$ and $|x|\leq N$, the diffusion coefficient $\nabla_{\theta}f$ satisfies
\[
\left|\nabla_{\theta}f(x,\theta_{1})-\nabla_{\theta}f(x,\theta_{2})\right|\leq C(N)|\theta_{1}-\theta_{2}|.
\]
Moreover, there exists $K>0$ and  $q>0$ such that
\[
|\nabla_{\theta}f(x,\theta)|\leq K (1+|x|^{q}).
\]
\item The function $f^{\ast}(x)$ is  $C^{2+\alpha}(\mathcal{X})$ with $\alpha\in(0,1)$. Namely, it has two derivatives in $x$, with all partial derivatives being H\"{o}lder continuous, with exponent $\alpha$, with respect to $x$.
\end{enumerate}
\end{condition}

Condition \ref{A:Assumption1} allows one to control the ergodic behavior of the X process.  As will be seen from the proof of the main convergence result Theorem \ref{T:MainTheorem}, one needs to control terms of the form $\int_{0}^{t}\alpha_{t}(\nabla \bar{g}(\theta_{s})-g(X_{s},\theta_{s}))ds$. Due to ergodicity of the $X$ process one expects that such terms are small in magnitude and go to zero as $t\rightarrow\infty$. However, the speed at which they go to zero is what matters here. We treat such terms by rewriting them equivalently using appropriate Poisson type partial differential equations (PDE). Condition \ref{A:Assumption1} guarantees that these Poisson equations have unique solutions that do not grow faster than polynomially in the $x$ variable (see Theorem \ref{T:RegularityPoisson} in Appendix \ref{S:RegularityResults}).

The main result of this paper is Theorem \ref{T:MainTheorem}.
\begin{theorem}\label{T:MainTheorem}
Assume that Conditions \ref{A:ExtraAssumption}, \ref{A:LyapunovCondition} and \ref{A:Assumption1} hold. Then we have that
\[
\lim_{t\rightarrow\infty}\|\nabla \bar{g}(\theta_{t})\|=0,\text{ almost surely}.
\]
\end{theorem}

\section{Proof of Theorem \ref{T:MainTheorem}}\label{S:ProofMainTheorem}
 We proceed in a spirit similar to that of \cite{BertsekasThitsiklis2000}. However, apart from continuous versus discrete dynamics, one of the main  challenges of the proof here is the presence of the ergodic $X$ process. Let us consider an arbitrarily given $\kappa>0$ and $\lambda=\lambda(\kappa)>0$ to be chosen. Then set $\sigma_{0}=0$ and consider the cycles of random times
\begin{eqnarray*}
0=\sigma_0 \leq \tau_1 \leq \sigma_1 \leq \tau_2 \leq \sigma_2 \leq \dots
\end{eqnarray*}
where for $k=1,2,\cdots$
\begin{align*}
\tau_{k}&=\inf\{t>\sigma_{k-1}:  \|\nabla \bar{g}(\theta_{t})\|\geq \kappa\},\nonumber\\
\sigma_{k}&=\sup\{t>\tau_{k}:  \frac{\|\nabla \bar{g}(\theta_{\tau_{k}})\|}{2}\leq \|\nabla \bar{g}(\theta_{s})\|\leq 2\|\nabla \bar{g}(\theta_{\tau_{k}})\| \text{ for all }s\in[\tau_{k},t] \text{ and }\int_{\tau_{k}}^{t}\alpha_{s}ds\leq \lambda\}.
\end{align*}

The purpose of these random times is to control the periods of time where $\|\nabla \bar{g}(\theta_{\cdot})\|$ is close to zero and away from zero. Let us next define the random time intervals $J_{k}=[\sigma_{k-1},\tau_{k})$ and $I_{k}=[\tau_{k},\sigma_{k})$. Notice that for every $t\in J_{k}$ we have $\|\nabla \bar g(\theta_{t})\|< \kappa$.

Let us next consider some $\eta>0$ sufficiently small to be chosen later on and set $\sigma_{k,\eta}=\sigma_{k}+\eta$.
 Lemma \ref{L:BoundForExtraTerm1} is crucial for the proof of Theorem \ref{T:MainTheorem}.
\begin{lemma}\label{L:BoundForExtraTerm1}
Assume that Conditions \ref{A:ExtraAssumption}, \ref{A:LyapunovCondition} and \ref{A:Assumption1} hold. Let us set
\[
\Gamma_{k,\eta}=\int_{\tau_{k}}^{\sigma_{k,\eta}}\alpha_{s}\left(\nabla_{\theta} g(X_{s},\theta_{s})-\nabla_{\theta} \bar{g}(\theta_{s})\right)ds.
\]

Then, with probability one we have that
\begin{align*}
\left\|\Gamma_{k,\eta}\right\|\rightarrow 0, \text{ as }k\rightarrow\infty.
\end{align*}
\end{lemma}

\begin{proof}
The idea is to use Theorem \ref{T:RegularityPoisson} in order to get an equivalent expression for the term $\Gamma_{k,\eta}$ that we seek to control.

Let us consider the function $G(x,\theta)=\nabla_{\theta} g(x,\theta)-\nabla_{\theta} \bar{g}(\theta)$. Notice that by definition and due to Condition \ref{A:Assumption1}, the function $G(x,\theta)$ satisfies the centering condition (\ref{Eq:CenteringCondition}) of Theorem \ref{T:RegularityPoisson} componentwise. So, the Poisson equation (\ref{Eq:CellProblem})
will have a unique smooth solution, denoted by $v(x,\theta)$ that grows at most polynomially in $x$. Let us apply It\^{o} formula to the vector valued function $u(t,x,\theta)=\alpha_{t}v(x,\theta)$. Doing so, we get for $i=1,\cdots, n$

\begin{align*}
u_{i}(\sigma,X_{\sigma},\theta_{\sigma})-u_{i}(\tau,X_{\tau},\theta_{\tau})&=\int_{\tau}^{\sigma}\partial_{s}u_{i}(s,X_{s},\theta_{s})ds
+\int_{\tau}^{\sigma}\mathcal{L}_{x}u_{i}(s,X_{s},\theta_{s})ds+\int_{\tau}^{\sigma}\mathcal{L}_{\theta}u_{i}(s,X_{s},\theta_{s})ds\nonumber\\
&+\int_{\tau}^{\sigma}\alpha_{s}\text{tr}\left[\nabla_{\theta}f(X_{s},\theta_{s})\nabla_{x}\nabla_{\theta}u_{i}(s,X_{s},\theta_{s})\right]ds\nonumber\\
&+\int_{\tau}^{\sigma}\left<\nabla_{x}u_{i}(s,X_{s},\theta_{s}),\sigma dW_{s}\right> +\int_{\tau}^{\sigma}\alpha_{s}\left<\nabla_{\theta}u_{i}(s,X_{s},\theta_{s}),\nabla_{\theta}f(X_{s},\theta_{s})\sigma^{-1}dW_{s}\right>,
\end{align*}
where $\mathcal{L}_{x}$ and $\mathcal{L}_{\theta}$ denote the infinitesimal generators for processes $X$ and $\theta$ respectively.

Recall now that $v(x,\theta)$ is the solution to the given Poisson equation and that  $u(s,x,\theta)=\alpha_{s}v(x,\theta)$. Using these facts and rearranging the previous It\^{o} formula, we get in vector notation
\begin{align}
\Gamma_{k,\eta}&=\int_{\tau_{k}}^{\sigma_{k,\eta}}\alpha_{s}\left(\nabla_{\theta} g(X_{s},\theta_{s})-\nabla_{\theta} \bar{g}(\theta_{s})\right)ds=\int_{\tau_{k}}^{\sigma_{k,\eta}}\mathcal{L}_{x}u(s,X_{s},\theta_{s})ds\nonumber\\
&=\left[\alpha_{\sigma_{k,\eta}}v(X_{\sigma_{k,\eta}},\theta_{\sigma_{k,\eta}})-\alpha_{\tau_{k}}v(X_{\tau_{k}},\theta_{\tau_{k}})-\int_{\tau_{k}}^{\sigma_{k,\eta}}\partial_{s}\alpha_{s} v(X_{s},\theta_{s})ds\right]\nonumber\\
&\quad-\int_{\tau_{k}}^{\sigma_{k,\eta}}\alpha_{s}\left[\mathcal{L}_{\theta}v (X_{s},\theta_{s})+\alpha_{s} \text{tr}\left[\nabla_{\theta}f(X_{s},\theta_{s})\nabla_{x_{i}}\nabla_{\theta}v(X_{s},\theta_{s})\right]_{i=1}^{m}\right]ds\nonumber\\
&\quad-\int_{\tau_{k}}^{\sigma_{k,\eta}}\alpha_{s}\left<\nabla_{x}v(X_{s},\theta_{s}),\sigma dW_{s}\right> -\int_{\tau_{k}}^{\sigma_{k,\eta}}\alpha_{s}^{2}\left<\nabla_{\theta}v(X_{s},\theta_{s}),\nabla_{\theta}f(X_{s},\theta_{s})\sigma^{-1}dW_{s}\right>.\label{Eq:EquivalentExpression1}
\end{align}

The next step is to treat each term on the right hand side of (\ref{Eq:EquivalentExpression1}) separately. For this purpose, let us first set
\[
J_{t}^{(1)}=\alpha_{t}\sup_{s\in[0,t]}\left\|v(X_{s},\theta_{s})\right\|.
\]

By Theorem \ref{T:RegularityPoisson} and Proposition 2 of \cite{PardouxVeretennikov1} there is some $0<K<\infty$ (that may change from line to line below) and $0<q<\infty$ such that for $t$ large enough
\begin{align*}
\mathbb{E}|J_{t}^{(1)}|^{2}&\leq K \alpha_{t}^{2}\mathbb{E}\left[1+\sup_{s\in[0,t]}\|X_{s}\|^{q}\right]= K \alpha_{t}^{2}\left[1+\sqrt{t}\frac{\mathbb{E}\sup_{s\in[0,t]}\|X_{s}\|^{q}}{\sqrt{t}}\right]\nonumber\\
&\leq K \alpha_{t}^{2}\left[1+\sqrt{t}\right]\leq K \alpha_{t}^{2}\sqrt{t}.
\end{align*}

By Condition \ref{A:ExtraAssumption} let us  consider $p>0$ such that $\lim_{t\rightarrow\infty}\alpha_{t}^{2}t^{1/2+2p}=0$ and for any $\delta\in(0,p)$ define the event $A_{t,\delta}=\left\{J_{t}^{(1)}\geq t^{\delta-p}\right\}$. Then we have for $t$ large enough such that $\alpha_{t}^{2}t^{1/2+2p}\leq 1$
\[
\mathbb{P}\left(A_{t,\delta}\right)\leq \frac{\mathbb{E}|J_{t}^{(1)}|^{2}}{t^{2(\delta-p)}}\leq K\frac{\alpha_{t}^{2}t^{1/2+2p}}{t^{2\delta}}\leq K\frac{1}{t^{2\delta}}.
\]

The latter implies that
\[
\sum_{n\in\mathbb{N}}\mathbb{P}\left(A_{2^{n},\delta}\right)<\infty.
\]

Therefore, by Borel-Cantelli lemma we have that for every $\delta\in(0,p)$ there is a finite positive random variable $d(\omega)$ and some $n_{0}<\infty$ such that for every $n\geq n_{0}$ one has
\[
J_{2^{n}}^{(1)}\leq\frac{d(\omega)}{2^{n(p-\delta)}}.
\]

Thus for $t\in[2^{n},2^{n+1})$ and $n\geq n_{0}$ one has for some finite constant $K<\infty$
\[
J_{t}^{(1)} \leq K \alpha_{2^{n+1}} \sup_{s\in(0,2^{n+1}]}\|v(X_{s},\theta_{s})\| \leq K \frac{d(\omega)}{2^{(n+1)(p-\delta)}}\leq K \frac{d(\omega)}{t^{p-\delta}}.
\] 

The latter display then guarantees that for $t\geq 2^{n_{0}}$ we have with probability one
\begin{align}
J_{t}^{(1)}\leq K \frac{d(\omega)}{t^{p-\delta}}\rightarrow 0, \text{ as }t\rightarrow\infty.\label{Eq:TermJ1}
\end{align}

Next we consider the term
\begin{align*}
J_{t,0}^{(2)}&=\int_{0}^{t}\left\|\alpha^{\prime}_{s} v(X_{s},\theta_{s})+\alpha_{s}\left(\mathcal{L}_{\theta}v (X_{s},\theta_{s})+\alpha_{s} \text{tr}\left[\nabla_{\theta}f(X_{s},\theta_{s})\nabla_{x_{i}}\nabla_{\theta}v(X_{s},\theta_{s})\right]_{i=1}^{m}\right)\right\|ds.
\end{align*}

By the bounds of Theorem \ref{T:RegularityPoisson} we see that there are constants $0<K<\infty$ (that may change from line to line) and $0<q<\infty$ such that
\begin{align*}
\sup_{t>0}\mathbb{E}|J_{t,0}^{(2)}|&\leq K \int_{0}^{\infty}(|\alpha^{\prime}_{s}|+\alpha_{s}^{2})(1+\mathbb{E}\|X_{s}\|^{q})ds\nonumber\\
&\leq K \int_{0}^{\infty}(|\alpha^{\prime}_{s}|+\alpha_{s}^{2})ds.\nonumber\\
&\leq K.
\end{align*}

The first inequality follows by Theorem \ref{T:RegularityPoisson}, the second inequality follows by Proposition 1 in \cite{PardouxVeretennikov1} and the third inequality follows by Condition \ref{A:ExtraAssumption}.

The latter display implies that  there is a finite random variable $\bar{J}_{\infty,0}^{(2)}$ such that
\begin{align}
J_{t,0}^{(2)}\rightarrow \bar{J}_{\infty,0}^{(2)}, \text{ as }t\rightarrow\infty \text{ with probability one}.\label{Eq:TermJ2}
\end{align}

The last term that we need to consider is the martingale term

\[
J_{t,0}^{(3)}=\int_{0}^{t}\alpha_{s}\left<\nabla_{x}v(X_{s},\theta_{s}),\sigma dW_{s}\right> +\int_{0}^{t}\alpha_{s}^{2}\left<\nabla_{\theta}v(X_{s},\theta_{s}),\nabla_{\theta}f(X_{s},\theta_{s})\sigma^{-1}dW_{s}\right>.
\]
Notice that the Burkholder-Davis-Gundy inequality and the bounds of Theorem \ref{T:RegularityPoisson} (doing calculations similar to the ones for the term $J_{t,0}^{(2)}$) give us that for some finite constant $K<\infty$, we have
\begin{align*}
\sup_{t>0}\mathbb{E}\left|J_{t,0}^{(3)}\right|^{2}\leq K \int_{0}^{\infty}\alpha_{s}^{2}ds< \infty
\end{align*}

Thus, by Doob's martingale convergence theorem there is a square integrable random variable $\bar{J}_{\infty,0}^{(3)}$ such that
\begin{align}
J_{t,0}^{(3)}\rightarrow \bar{J}_{\infty,0}^{(3)}, \text{ as }t\rightarrow\infty \text{ both almost surely and in } L^{2}. \label{Eq:TermJ3}
\end{align}

Let us now go back to (\ref{Eq:EquivalentExpression1}). Using the terms $J_{t}^{(1)}$, $J_{t,0}^{(2)}$ and $J_{t,0}^{(3)}$ we can write
\begin{align*}
\|\Gamma_{k,\eta}\|&\leq J_{\sigma_{k,\eta}}^{(1)}+J_{\tau_{k}}^{(1)}+J_{\sigma_{k,\eta},\tau_{k}}^{(2)}+\|J_{\sigma_{k,\eta},\tau_{k}}^{(3)}\|
\end{align*}

The last display together with (\ref{Eq:TermJ1}), (\ref{Eq:TermJ2}) and (\ref{Eq:TermJ3}) imply the statement of the lemma.
\end{proof}

\begin{lemma}\label{L:Interval}
Assume that Conditions \ref{A:ExtraAssumption}, \ref{A:LyapunovCondition} and \ref{A:Assumption1} hold. Choose $\lambda>0$ such that for a given $\kappa>0$, one has $3\lambda+\frac{\lambda}{4\kappa}=\frac{1}{2 L_{\nabla \bar{g}}}$, where $L_{\nabla \bar{g}}$ is the Lipschitz constant of $\nabla \bar{g}$. For $k$ large enough and for $\eta>0$ small enough (potentially random depending on $k$), one has $\int_{\tau_{k}}^{\sigma_{k,\eta}}\alpha_{s}ds>\lambda$. In addition we also have
$\frac{\lambda}{2}\leq\int_{\tau_{k}}^{\sigma_{k}}\alpha_{s}ds\leq\lambda$ with probability one.
\end{lemma}
\begin{proof}
Let us define the random variable
\begin{align}
R_{s}=\sum_{k\geq 1}\|\nabla \bar{g}(\theta_{\tau_{k}})\|1_{s\in I_{k}}+\kappa 1_{s\in[0,\infty)\setminus \bigcup_{k\geq 1} I_{k}}.\label{Eq:Rs}
\end{align}
 Then, for any $s\in\mathbb{R}$ we have $\|\nabla \bar{g}(\theta_{s})\|/R_{s}\leq 2$.

 We proceed with an argument via contradiction. In particular let us assume that $\int_{\tau_{k}}^{\sigma_{k,\eta}}\alpha_{s}ds\leq \lambda$ and let us choose arbitrarily some $\epsilon>0$ such that $\epsilon\leq \lambda/8$.

 Let us now make some remarks that are independent of the sign of $\int_{\tau_{k}}^{\sigma_{k,\eta}}\alpha_{s}ds-\lambda$. Due to the summability condition $\int_{0}^{\infty}\alpha^{2}_{t}dt<\infty$, $\frac{\kappa}{\|\nabla \bar{g}(\theta_{\tau_{k}})\|}\leq 1$ and Conditions \ref{A:ExtraAssumption} and \ref{A:Assumption1}, we have that
\[
\sup_{t>0}\mathbb{E}\left|\int_{0}^{t}\alpha_{s}\frac{\kappa}{\|\nabla \bar{g}(\theta_{\tau_{k}})\|}\nabla_{\theta}f(X_{s},\theta_{s})\sigma^{-1} dW_{s}\right|^{2}<\infty
  \]

Hence, the martingale convergence theorem applies to the martingale $\int_{0}^{t}\alpha_{s}\frac{\kappa}{\|\nabla \bar{g}(\theta_{\tau_{k}})\|}\nabla_{\theta}f(X_{s},\theta_{s})\sigma^{-1} dW_{s}$. This means that there exists a square integrable random variable $M$ such that $\int_{0}^{t}\alpha_{s}\frac{\kappa}{\|\nabla \bar{g}(\theta_{\tau_{k}})\|}\nabla_{\theta}f(X_{s},\theta_{s})\sigma^{-1} dW_{s}\rightarrow M$ both almost surely and in $L^{2}$. This means that for the given $\epsilon>0$ there is $k$ large enough such that  $\left\|\int_{\tau_{k}}^{\sigma_{k,\eta}}\alpha_{s}\frac{\kappa}{\|\nabla \bar{g}(\theta_{\tau_{k}})\|}\nabla_{\theta}f(X_{s},\theta_{s})\sigma^{-1}dW_{s}\right\|<\epsilon$ almost surely.

 Let us also assume that for the given $k$, $\eta$ is so small such that for any $s\in[\tau_{k},\sigma_{k,\eta}]$ one has $\|\nabla \bar{g}(\theta_{s})\|\leq 3\|\nabla \bar{g}(\theta_{\tau_{k}})\|$.

 Then, we obtain the following
\begin{align*}
\left\|\theta_{\sigma_{k,\eta}}-\theta_{\tau_{k}}\right\|&=\left\|-\int_{\tau_{k}}^{\sigma_{k,\eta}}\alpha_{s}\nabla_{\theta} g(X_{s},\theta_{s})ds+\int_{\tau_{k}}^{\sigma_{k,\eta}}\alpha_{s}\nabla_{\theta} f(X_{s},\theta_{s})\sigma^{-1}dW_{s}\right\|\nonumber\\
&=\left\|-\int_{\tau_{k}}^{\sigma_{k,\eta}}\alpha_{s}\nabla_{\theta} \bar{g}(\theta_{s})ds-\int_{\tau_{k}}^{\sigma_{k,\eta}}\alpha_{s}\left(\nabla_{\theta} g(X_{s},\theta_{s})-\nabla_{\theta} \bar{g}(\theta_{s})\right)ds+\int_{\tau_{k}}^{\sigma_{k,\eta}}\alpha_{s}\nabla_{\theta} f(X_{s},\theta_{s})\sigma^{-1}dW_{s}\right\|\nonumber\\
&\leq \int_{\tau_{k}}^{\sigma_{k,\eta}}\alpha_{s}\left\|\nabla \bar{g}(\theta_{s})\right\|ds+\left\|\int_{\tau_{k}}^{\sigma_{k,\eta}}\alpha_{s}\left(\nabla_{\theta} g(X_{s},\theta_{s})-\nabla_{\theta} \bar{g}(\theta_{s})\right)ds\right\|+ \left\|\int_{\tau_{k}}^{\sigma_{k,\eta}}\alpha_{s}\nabla_{\theta} f(X_{s},\theta_{s})\sigma^{-1}dW_{s}\right\|\nonumber\\
&\leq  3\|\nabla \bar{g}(\theta_{\tau_{k}})\|\int_{\tau_{k}}^{\sigma_{k,\eta}}\alpha_{s}ds+\left\|\int_{\tau_{k}}^{\sigma_{k,\eta}}\alpha_{s}\left(\nabla_{\theta} g(X_{s},\theta_{s})-\nabla_{\theta} \bar{g}(\theta_{s})\right)ds\right\|\nonumber\\
&\qquad+\frac{\|\nabla \bar{g}(\theta_{\tau_{k}})\|}{\kappa}\left\|\int_{\tau_{k}}^{\sigma_{k,\eta}}\alpha_{s}\frac{\kappa}{\|\nabla \bar{g}(\theta_{\tau_{k}})\|}\nabla_{\theta} f(X_{s},\theta_{s})\sigma^{-1}dW_{s}\right\|\nonumber\\
&\leq  3\|\nabla \bar{g}(\theta_{\tau_{k}})\|\lambda+\left\|\int_{\tau_{k}}^{\sigma_{k,\eta}}\alpha_{s}\left(\nabla_{\theta} g(X_{s},\theta_{s})-\nabla_{\theta} \bar{g}(\theta_{s})\right)ds\right\|+\frac{\|\nabla \bar{g}(\theta_{\tau_{k}})\|}{\kappa}\epsilon\nonumber\\
&\leq  3\|\nabla \bar{g}(\theta_{\tau_{k}})\|\lambda+\left\|\int_{\tau_{k}}^{\sigma_{k,\eta}}\alpha_{s}\left(\nabla_{\theta} g(X_{s},\theta_{s})-\nabla_{\theta} \bar{g}(\theta_{s})\right)ds\right\|+\frac{\|\nabla \bar{g}(\theta_{\tau_{k}})\|}{\kappa}\lambda/8\nonumber\\
&=  \|\nabla \bar{g}(\theta_{\tau_{k}})\| \left[3\lambda+\frac{\lambda}{8\kappa}\right]+\left\|\int_{\tau_{k}}^{\sigma_{k,\eta}}\alpha_{s}\left(\nabla_{\theta} g(X_{s},\theta_{s})-\nabla_{\theta} \bar{g}(\theta_{s})\right)ds\right\|.
\end{align*}

Let us next bound appropriately the Euclidean norm of the vector-valued random variable
\[
\Gamma_{k,\eta}=\int_{\tau_{k}}^{\sigma_{k,\eta}}\alpha_{s}\left(\nabla_{\theta} g(X_{s},\theta_{s})-\nabla_{\theta} \bar{g}(\theta_{s})\right)ds.
\]

By Lemma \ref{L:BoundForExtraTerm1} we have that for the same $0<\epsilon<\lambda/8$ that was chosen before there is $k$ large enough such that almost surely
\[
\|\Gamma_{k,\eta}\|\leq \epsilon\leq\lambda/8.
\]

Hence, using also the fact that $\frac{\kappa}{\|\nabla \bar{g}(\theta_{\tau_{k}})\|}\leq 1$ we obtain
\begin{align*}
\left\|\theta_{\sigma_{k,\eta}}-\theta_{\tau_{k}}\right\|&\leq \|\nabla \bar{g}(\theta_{\tau_{k}})\| \left[3\lambda+\frac{\lambda}{4\kappa}\right]= \|\nabla \bar{g}(\theta_{\tau_{k}})\| \frac{1}{2 L_{\nabla \bar{g}} }.
\end{align*}

The latter then implies that we should  have
\begin{align*}
\|\nabla \bar{g}(\theta_{\sigma_{k,\eta}})-\nabla \bar{g}(\theta_{\tau_{k}})\|&\leq L_{\nabla \bar{g}}\left\|\theta_{\sigma_{k,\eta}}-\theta_{\tau_{k}}\right\|\leq \frac{\|\nabla \bar{g}(\theta_{\tau_{k}})\|}{2}.
\end{align*}

The latter statement will then imply that
\begin{align*}
\frac{\|\nabla \bar{g}(\theta_{\tau_{k}})\|}{2}\leq \|\nabla \bar{g}(\theta_{\sigma_{k,\eta}})\|\leq 2\|\nabla \bar{g}(\theta_{\tau_{k}})\|.
\end{align*}
But then we would necessarily have that $\int_{\tau_{k}}^{\sigma_{k,\eta}}\alpha_{s}ds>\lambda$, since otherwise $\sigma_{k,\eta}\in[\tau_{k},\sigma_{k}]$ which is impossible.

Next we move on to prove the second statement of the lemma. By definition we have $\int_{\tau_{k}}^{\sigma_{k}}\alpha_{s}ds\leq\lambda$. So it remains to show that $\frac{\lambda}{2}\leq\int_{\tau_{k}}^{\sigma_{k}}\alpha_{s}ds$. Since we know that $\int_{\tau_{k}}^{\sigma_{k,\eta}}\alpha_{s}ds>\lambda$ and because for $k$ large enough and $\eta$ small enough one should have $\int_{\sigma_{k}}^{\sigma_{k,\eta}}\alpha_{s}ds\leq \lambda/2$, we obtain that
\[
\int_{\tau_{k}}^{\sigma_{k}}\alpha_{s}ds\geq \lambda-\int_{\sigma_{k}}^{\sigma_{k,\eta}}\alpha_{s}ds\geq \lambda-\lambda/2=\lambda/2,
\]
concluding the proof of the lemma.
\end{proof}

Lemma \ref{L:BoundednessOfg} shows that the function $\bar{g}$ and its first two derivatives are uniformly bounded in $\theta$.
\begin{lemma}\label{L:BoundednessOfg} Assume Conditions \ref{A:ExtraAssumption}, \ref{A:LyapunovCondition} and \ref{A:Assumption1}. For any $q>0$, there is a constant $ K$ such that
\begin{align}
\int_\mathcal{X}(1+|x|^q)\pi(dx)\leq C.\nonumber
\end{align}
In addition we also have that there is a constant $C<\infty$ such that $\sum_{i=0}^{2}\|\nabla^{i}_{\theta}\bar{g}(\theta)\|\leq C$.
\end{lemma}

\begin{proof} By Theorem 1 in \cite{pardoux2003poisson}, the density $\mu$ of the measure $\pi$ admits, for any $p$, a constant $C_p$ such that $\mu(x)\leq\frac{C_p}{1+|x|^p}$.  Choosing $p$ large enough that $\int_\mathcal{X}\frac{1+|x|^q}{1+|x|^p}dy<\infty$, we then obtain

\begin{align}
\int_\mathcal{X}(1+|x|^q)\pi(dx)
&\leq\int_\mathcal{X}C_p\frac{1+|x|^q}{1+|x|^p}dx\leq \ C.\nonumber
\end{align}
concluding the proof of the first statement of the lemma. Let us now focus on the second part of the lemma. We only prove the claim for $i=0$, since due to the bounds in Condition \ref{A:Assumption1}, the proof for $i=1,2$ is the same. By Condition \ref{A:Assumption1} and by the first part of the lemma, we have that there exist constants $0<q,K ,C<\infty$ such that
\begin{align*}
\bar{g}(\theta)&=\int_{\mathcal{X}}\frac{1}{2}\left\|f(x,\theta)-f^{\ast}(x)\right\|^{2}\pi(dx)\leq K\int_\mathcal{X}(1+|x|^q)\pi(dx)\leq C,
\end{align*}
 concluding the proof of the lemma.
\end{proof}

Our next goal is to show that if the index $k$ is large enough, then $\bar{g}$ decreases, in the sense of Lemma \ref{LemmaUpperNegativeBound}.
\begin{lemma} \label{LemmaUpperNegativeBound}
Assume Conditions \ref{A:ExtraAssumption}, \ref{A:LyapunovCondition} and \ref{A:Assumption1}. Suppose that there are an infinite number of intervals $I_{k}=[\tau_k,\sigma_{k})$.  There is a fixed constant $\gamma=\gamma(\kappa)>0$ such that for $k$ large enough, one has
\begin{align}
\bar{g}(\theta_{\sigma_{k}})-\bar{g}(\theta_{\tau_{k}})\leq -\gamma.
\end{align}
\end{lemma}

\begin{proof}
By It\^{o}'s formula we have that
\begin{eqnarray*}
\bar g(\theta_{\sigma_{k}}) - \bar g (\theta_{\tau_{k}}) &=&  -  \int_{\tau_{k}}^{\sigma_{k}} \alpha_{s} \left\| \nabla \bar g (\theta_s) \right\|^2 ds + \int_{\tau_{k}}^{\sigma_{k}}  \alpha_{s}  \left<\nabla \bar g (\theta_s), \nabla_{\theta} f(X_s, \theta_s) \sigma^{-1}  d W_s\right>\notag\\
&+& \int_{\tau_{k}}^{\sigma_{k}} \frac{\alpha_{s}^{2}}{2} \text{tr}\left[(\nabla_{\theta} f(X_s, \theta_s) \sigma^{-1}) (\nabla_{\theta} f(X_s, \theta_s) \sigma^{-1})^{\top} \nabla_{\theta}\nabla_{\theta}  \bar g (\theta_s) \right]ds \notag \notag\\
&+&\int_{\tau_{k}}^{\sigma_{k}} \alpha_{s} \left<\nabla_{\theta} \bar g(\theta_{s}), \nabla_{\theta} \bar{g}(\theta_{s})- \nabla_{\theta} g(X_{s},\theta_{s})\right>ds   \notag \\
&=& \Theta_{1,k}+\Theta_{2,k}+\Theta_{3,k} + \Theta_{4,k}.
\end{eqnarray*}

Let's first consider $\Theta_{1,k}$.  Notice that for all $s\in[\tau_{k},\sigma_{k}]$ one has $ \frac{\|\nabla \bar{g}(\theta_{\tau_{k}})\|}{2}\leq \|\nabla \bar{g}(\theta_{s})\|\leq 2\|\nabla \bar{g}(\theta_{\tau_{k}})\| $. Hence, for sufficiently large $k$, we have the upper bound:
\begin{eqnarray*}
\Theta_{1,k} = -  \int_{\tau_{k}}^{\sigma_{k}} \alpha_{s} \left\| \nabla \bar g (\theta_s) \right\|^2 ds \leq  - \frac{ \left\|\nabla \bar g (\theta_{\tau_k})\right\|^2 }{4} \int_{\tau_{k}}^{\sigma_{k}} \alpha_{s} ds \leq -  \frac{ \left\|\nabla \bar g (\theta_{\tau_k})\right\|^2 }{8} \lambda,
\end{eqnarray*}
since Lemma \ref{L:BoundForExtraTerm1} proved that $\int_{\tau_{k}}^{\sigma_{k}} \alpha_{s} ds \geq \frac{\lambda}{2}$ for sufficiently large $k$.

We next address $\Theta_{2,k}$ and show that it becomes small as $k \rightarrow \infty$. First notice that we can trivially write
\begin{align*}
\Theta_{2,k}&=\int_{\tau_{k}}^{\sigma_{k}}  \alpha_{s}  \left<\nabla \bar g (\theta_s),\nabla_{\theta} f(X_s, \theta_s) \sigma^{-1}  d W_s\right> = \|\nabla \bar g(\theta_{\tau_{k}})\|\int_{\tau_{k}}^{\sigma_{k}}  \alpha_{s} \left<\frac{\nabla \bar g (\theta_s) }{\|\nabla \bar g(\theta_{\tau_{k}})\|},\nabla_{\theta} f(X_s, \theta_s) \sigma^{-1}  d W_s \right> \notag \\
&=   \|\nabla \bar g(\theta_{\tau_{k}})\|\int_{\tau_{k}}^{\sigma_{k}}  \alpha_{s} \left<\frac{\nabla \bar g (\theta_s) }{R_s},\nabla_{\theta} f(X_s, \theta_s)\sigma^{-1}   d W_s\right>.
\end{align*}

By Condition \ref{A:Assumption1} and It\^{o} isometry we have
\begin{align*}
\sup_{t>0}\mathbb{E}\left|\int_{0}^{t}  \alpha_{s} \left<\frac{ \nabla \bar g(\theta_s)}{R_{s}},\nabla_{\theta} f(X_s, \theta_s)  \sigma^{-1} d W_s\right>\right|^{2}&\leq 4 \mathbb{E} \int_{0}^{\infty}  \alpha^{2}_{s} \norm{ \nabla_{\theta} f(X_s, \theta_s) }^2 ds\nonumber\\
&\leq K  \int_{0}^{\infty}  \alpha^{2}_{s} \left(1+\mathbb{E}\|X_{s}\|^{q}\right) ds< \infty,
\end{align*}
where $R_{s}$ is defined via (\ref{Eq:Rs}). Hence, by Doob's martingale convergence theorem there is a square integrable random variable $M$ such that $\int_{0}^{t}  \alpha_{s} \left<\frac{ \nabla \bar g(\theta_s)}{R_{s}},\nabla_{\theta} f(X_s, \theta_s)   d W_s\right>\rightarrow M$ both almost surely and in $L^{2}$.
The latter statement implies that for a given $\epsilon>0$ there is $k$  large enough such that almost surely
\begin{align*}
\Theta_{2,k}&\leq \|\nabla \bar g(\theta_{\tau_{k}})\|\epsilon.\label{Eq:TermTheta2}
\end{align*}

We now consider $\Theta_{3,k}$.
\begin{eqnarray}
&\phantom{.}& \sup_{t>0}\mathbb{E}  \left\| \int_{0}^{t} \frac{\alpha_{s}^{2}}{2} \text{tr}\left[(\nabla_{\theta} f(X_s, \theta_s) \sigma^{-1})(\nabla_{\theta} f(X_s, \theta_s)\sigma^{-1})^{\top}  \nabla_{\theta}\nabla_{\theta}  \bar g (\theta_s) \right] ds  \right\|  \notag \\
 &  \leq &   C  \int_{0}^{\infty} \frac{\alpha_{s}^{2}}{2}\mathbb{E}\left(1+\|X_{s}\|^{q}\right)  ds <\infty,
 \label{Theta3InfinityA}
\end{eqnarray}
where we have used Condition \ref{A:Assumption1} and Lemma \ref{L:BoundednessOfg}.  Bound (\ref{Theta3InfinityA}) implies that
\[
\int_{0}^{\infty} \frac{\alpha_{s}^{2}}{2} \text{tr}\left[(\nabla_{\theta} f(X_s, \theta_s) \sigma^{-1})(\nabla_{\theta} f(X_s, \theta_s)\sigma^{-1})^{\top}  \nabla_{\theta}\nabla_{\theta}  \bar g (\theta_s) \right] ds
 \]
 is finite almost surely, which in turn implies that there is a finite random variable $\Theta_3^{\infty}$ such that
 \[
 \int_{0}^{t} \frac{\alpha_{s}^{2}}{2} \text{tr}\left[(\nabla_{\theta} f(X_s, \theta_s) \sigma^{-1})(\nabla_{\theta} f(X_s, \theta_s)\sigma^{-1})^{\top} \nabla_{\theta}\nabla_{\theta}  \bar g (\theta_s) \right] ds \rightarrow \Theta_3^{\infty} \text{ as } t \rightarrow \infty,
  \]
  with probability one.  Since $\Theta_3^{\infty}$ is finite, $\int_{\tau_k}^{\sigma_k} \frac{\alpha_{s}^{2}}{2} \text{tr}\left[(\nabla_{\theta} f(X_s, \theta_s) \sigma^{-1})(\nabla_{\theta} f(X_s, \theta_s)\sigma^{-1})^{\top}
   \nabla_{\theta}\nabla_{\theta}  \bar g (\theta_s) \right] ds \rightarrow 0$ as $k \rightarrow \infty$ with probability one.

Finally, we address $ \Theta_{4,k} $.  Let us consider the function $G(x,\theta)= \left<\nabla_{\theta} \bar g(\theta), \nabla_{\theta} g(x,\theta)-\nabla_{\theta} \bar{g}(\theta)\right>$.  The function $G(x,\theta)$ satisfies the centering condition (\ref{Eq:CenteringCondition}) of Theorem \ref{T:RegularityPoisson}. Therefore, the Poisson equation (\ref{Eq:CellProblem}) with right hand side $G(x,\theta)$ will have a unique smooth solution, say $v(x,\theta)$, that grows at most polynomially in $x$. Let us apply It\^{o} formula to the function $u(t,x,\theta)=\alpha_{t}v(x,\theta)$ that is solution to this Poisson equation.

\begin{align*}
u(\sigma,X_{\sigma},\theta_{\sigma})-u(\tau,X_{\tau},\theta_{\tau})&=\int_{\tau}^{\sigma}\partial_{s}u(s,X_{s},\theta_{s})ds
+\int_{\tau}^{\sigma}\mathcal{L}_{x}u(s,X_{s},\theta_{s})ds+\int_{\tau}^{\sigma}\mathcal{L}_{\theta}u(s,X_{s},\theta_{s})ds\nonumber\\
&+\int_{\tau}^{\sigma}\alpha_{s}\text{tr}\left[\nabla_{\theta}f(X_{s},\theta_{s})\nabla_{x}\nabla_{\theta}u(s,X_{s},\theta_{s})\right]ds\nonumber\\
&+\int_{\tau}^{\sigma}\left< \nabla_{x}u(s,X_{s},\theta_{s}),\sigma dW_{s}\right>+\int_{\tau}^{\sigma}\alpha_{s}\left<\nabla_{\theta}u(s,X_{s},\theta_{s}),\nabla_{\theta}f(X_{s},\theta_{s})\sigma^{-1}dW_{s}\right>.
\end{align*}

Rearranging the previous It\^{o} formula yields
\begin{align*}
\Theta_{4,k} &=\int_{\tau_{k}}^{\sigma_{k}}\alpha_{s}  \left<\nabla_{\theta} \bar g(\theta_{t}),   \nabla_{\theta} g(X_{s},\theta_{s})-\nabla_{\theta} \bar{g}(\theta_{s})\right>ds=\int_{\tau_{k}}^{\sigma_{k}}\mathcal{L}_{x}u(s,X_{s},\theta_{s})ds\nonumber\\
&=\left[\alpha_{\sigma_{k}}v(X_{\sigma_{k}},\theta_{\sigma_{k}})-\alpha_{\tau_{k}}v(X_{\tau_{k}},\theta_{\tau_{k}})-\int_{\tau_{k}}^{\sigma_{k}}\partial_{s}\alpha_{s} v(X_{s},\theta_{s})ds\right]\nonumber\\
&\quad-\int_{\tau_{k}}^{\sigma_{k}}\alpha_{s}\left[\mathcal{L}_{\theta}v (X_{s},\theta_{s})+\alpha_{s} \text{tr}\left[\nabla_{\theta}f(X_{s},\theta_{s})\nabla_{x}\nabla_{\theta}v(X_{s},\theta_{s})\right]\right]ds\nonumber\\
&\quad-\int_{\tau_{k}}^{\sigma_{k}}\alpha_{s}\left<\nabla_{x}v(X_{s},\theta_{s}),\sigma dW_{s}\right> -\int_{\tau_{k}}^{\sigma_{k}}\alpha_{s}\left<\nabla_{\theta}v(X_{s},\theta_{s}),\nabla_{\theta}f(X_{s},\theta_{s})\sigma^{-1}dW_{s}\right>.\label{Eq:EquivalentExpression1}
\end{align*}

Following the exact same steps as in the proof of Lemma \ref{L:BoundForExtraTerm1} gives us that $\lim_{k \rightarrow \infty} \norm{ \Theta_{4,k} } \rightarrow 0$ almost surely.

We now return to $\bar g(\theta_{\sigma_{k}}) - \bar g (\theta_{\tau_{k}})$ and provide an upper bound which is negative.  For sufficiently large $k$, we have that:
\begin{eqnarray*}
\bar g(\theta_{\sigma_{k}}) - \bar g (\theta_{\tau_{k}})  &\leq& -\frac{\|\nabla \bar g(\theta_{\tau_{k}})\|^{2}}{8} \lambda  + \norm{\Theta_{2,k}} + \norm{\Theta_{3,k} } + \norm{\Theta_{4,k} }  \notag \\
&\leq& -\frac{\|\nabla \bar g(\theta_{\tau_{k}})\|^{2}}{8} \lambda + \|\nabla \bar g(\theta_{\tau_{k}})\| \epsilon+ \epsilon + \epsilon.
\end{eqnarray*}
Choose $\epsilon = \min \{ \frac{ \lambda \kappa^2} {32}, \frac{\lambda}{32} \}$.  On the one hand, if $\|\nabla \bar g(\theta_{\tau_{k}})\| \geq 1$:
\begin{eqnarray*}
\bar g(\theta_{\sigma_{k}}) - \bar g (\theta_{\tau_{k}})   &\leq& -\frac{\|\nabla \bar g(\theta_{\tau_{k}})\|^{2}}{8} \lambda + \|\nabla \bar g(\theta_{\tau_{k}})\|^2 \epsilon+ \epsilon + \epsilon \notag \\
&\leq& -3 \frac{\|\nabla \bar g(\theta_{\tau_{k}})\|^{2}}{32} \lambda +2 \epsilon \leq  -3 \frac{ \kappa^2 }{32} \lambda + 2 \frac{ \kappa^2 }{32} \lambda \leq  - \frac{ \kappa^2 }{32} \lambda .
\end{eqnarray*}
On the other hand, if $\|\nabla \bar g (\theta_{\tau_{k}})\| \leq 1$, then
\begin{eqnarray*}
\bar g(\theta_{\sigma_{k}}) - \bar g (\theta_{\tau_{k}})   &\leq& -\frac{\|\nabla \bar g(\theta_{\tau_{k}})\|^{2}}{8} \lambda + \epsilon+ \epsilon + \epsilon \notag \\
&\leq& - \frac{4 \kappa^2}{32} \lambda +3 \epsilon \leq  -4 \frac{ \kappa^2 }{32} \lambda + 3 \frac{ \kappa^2 }{32} \lambda \leq  - \frac{ \kappa^2 }{32} \lambda .
\end{eqnarray*}
Finally, let $\gamma = \frac{ \kappa^2 }{32} \lambda$ and the proof of the lemma is complete.

\end{proof}

\begin{lemma} \label{JintervalUpperBound}
Assume Conditions \ref{A:ExtraAssumption}, \ref{A:LyapunovCondition} and \ref{A:Assumption1}. Suppose that there are an infinite number of intervals $I_{k}=[\tau_k,\sigma_{k})$.  There is a fixed constant $\gamma_1 < \gamma$ such that for $k$ large enough,
\begin{align*}
\bar{g}(\theta_{\tau_{k}})-\bar{g}(\theta_{\sigma_{k-1}}) \leq \gamma_1.
\end{align*}
\end{lemma}

\begin{proof}
First, recall that $ \norm{ \nabla \bar g (\theta_t) } \leq \kappa $ for $ t \in J_k = [ \sigma_{k-1}, \tau_k ]$.  Similar to before, we have that:
\begin{eqnarray}
\bar g(\theta_{\tau_{k}}) - \bar g (\theta_{\sigma_{k-1}}) &=&  -  \int_{\sigma_{k-1}}^{\tau_k} \alpha_{s} \left\| \nabla \bar g (\theta_s) \right\|^2 ds + \int_{\sigma_{k-1}}^{\tau_{k}}  \alpha_{s}  \left<\nabla \bar g (\theta_s),\nabla_{\theta} f(X_s, \theta_s) \sigma^{-1}  d W_s\right>\nonumber\\
&+& \int_{\sigma_{k-1}}^{\tau_{k}} \frac{\alpha_{s}^{2}}{2} \text{tr}\left[ (\nabla_{\theta} f(X_s, \theta_s)\sigma^{-1})(\nabla_{\theta} f(X_s, \theta_s)\sigma^{-1})^{\top}  \nabla^{2}_{\theta}  \bar g (\theta_s) \right]ds \notag \\
&+& \int_{\sigma_{k-1}}^{\tau_{k}} \alpha_{s} \left<\nabla_{\theta} \bar g(\theta_{s}), \nabla_{\theta} \bar{g}(\theta_{s})- \nabla_{\theta} g(X_{s},\theta_{s})\right>ds   \notag \\
& \leq &  \int_{\sigma_{k-1}}^{\tau_{k}}  \alpha_{s}  \left<\nabla \bar g (\theta_s),\nabla_{\theta} f(X_s, \theta_s)\sigma^{-1}   d W_s\right> \nonumber\\
 &+&\int_{\sigma_{k-1}}^{\tau_{k}} \frac{\alpha_{s}^{2}}{2} \text{tr}\left[ (\nabla_{\theta} f(X_s, \theta_s)\sigma^{-1})(\nabla_{\theta} f(X_s, \theta_s)\sigma^{-1})^{\top} \nabla^{2}_{\theta}  \bar g (\theta_s) \right] ds \notag \\
&+& \int_{\sigma_{k-1}}^{\tau_{k}} \alpha_{s} \left<\nabla_{\theta} \bar g(\theta_{s}), \nabla_{\theta} \bar{g}(\theta_{s})- \nabla_{\theta} g(X_{s},\theta_{s})\right>ds.
\label{EqnUpperBoundJ}
\end{eqnarray}
The right hand side (RHS) of equation (\ref{EqnUpperBoundJ}) converges almost surely to $0$ as $k \rightarrow \infty$ as a consequence of similar arguments as given in Lemma \ref{LemmaUpperNegativeBound}. Indeed, the treatment of the second and third terms on the RHS of (\ref{EqnUpperBoundJ}) are exactly the same as in Lemma \ref{LemmaUpperNegativeBound}.  It remains to show that the first term on the RHS of (\ref{EqnUpperBoundJ}) converges almost surely to $0$ as $k \rightarrow \infty$.

\begin{eqnarray}
\int_{\sigma_{k-1}}^{\tau_{k}}  \alpha_{s}  \left<\nabla \bar g (\theta_s),\nabla_{\theta} f(X_s, \theta_s)  \sigma^{-1} d W_s \right>&=& \norm{ \nabla \bar g (\theta_{\sigma_{k-1} }) } \int_{\sigma_{k-1}}^{\tau_{k}}  \alpha_{s}  \left<\frac{ \nabla \bar g (\theta_s) }{\norm{ \nabla \bar g (\theta_{\sigma_{k-1} }) }},\nabla_{\theta} f(X_s, \theta_s) \sigma^{-1}  d W_s\right>    \notag \\
&=& \norm{ \nabla \bar g (\theta_{\sigma_{k-1} }) } \int_{\sigma_{k-1}}^{\tau_{k}}  \alpha_{s} \left<\frac{ \nabla \bar g (\theta_s) }{R_s}, \nabla_{\theta} f(X_s, \theta_s)  \sigma^{-1} d W_s\right>.
\label{MartIntJ}
\end{eqnarray}
As shown in Lemma \ref{LemmaUpperNegativeBound}, $\int_{\sigma_{k-1}}^{\tau_{k}}  \alpha_{s}   \left<\frac{ \nabla \bar g (\theta_s) }{R_s},\nabla_{\theta}f(X_s, \theta_s) \sigma^{-1}  d W_s\right> \rightarrow 0$ as $k \rightarrow \infty$ almost surely.  Finally, note that $\norm{ \nabla \bar g (\theta_{\sigma_{k-1} }) }  \leq \kappa$ (except when $\sigma_{k-1} = \tau_k$, in which case the interval $J_k$ is length 0 and hence the integral (\ref{MartIntJ}) over $J_k$ is $0$).  Then,
$\int_{\sigma_{k-1}}^{\tau_{k}}  \alpha_{s}   \left<\nabla \bar g (\theta_s),\nabla_{\theta}f(X_s, \theta_s) \sigma^{-1}  d W_s\right> \rightarrow 0$ as $k \rightarrow \infty$ almost surely.

Therefore, with probability one, $\bar g(\theta_{\tau_{k}}) - \bar g (\theta_{\sigma_{k-1}})  \leq \gamma_1 < \gamma$ for sufficiently large $k$.

\end{proof}

\begin{proof}[Proof of Theorem \ref{T:MainTheorem}]
Choose a $\kappa > 0$.  First, consider the case where there are a finite number of times $\tau_k$.  Then, there is a finite $T$ such that $\norm { \nabla \bar g (\theta_{t}) } < \kappa $ for $t \geq T$.  Now, consider the other case where there are an infinite number of times $\tau_k$ and use Lemmas \ref{LemmaUpperNegativeBound}  and \ref{JintervalUpperBound}.  With probability one,
\begin{eqnarray}
\bar g(\theta_{\sigma_{k}}) - \bar g (\theta_{\tau_{k}}) & \leq &  -  \gamma = - \frac{ \kappa^2 }{32} \lambda , \notag \\
\bar g(\theta_{\tau_{k}}) - \bar g (\theta_{\sigma_{k-1}}) & \leq &  \gamma_1 < \gamma ,
\label{BoundsGamma}
\end{eqnarray}
for sufficiently large $k$.  Choose a $K$ such that (\ref{BoundsGamma}) holds for $k \geq K$. This leads to:
\begin{eqnarray*}
\bar g(\theta_{\tau_{n+1}}) - \bar g (\theta_{\tau_{K}}) = \sum_{k = K}^n \bigg{[} \bar g(\theta_{\sigma_{k}}) - \bar g (\theta_{\tau_{k}}) + \bar g(\theta_{\tau_{k+1}}) - \bar g (\theta_{\sigma_{k}}) \bigg{]} \leq  \sum_{k = K}^n ( - \gamma + \gamma_1 )<0 .
\end{eqnarray*}
Let $n \rightarrow \infty$ and then $\bar g(\theta_{\tau_{n+1}}) \rightarrow - \infty$.  However, we also have that by definition $\bar g(\theta) \geq 0$.  This is a contradiction, and therefore almost surely there are a finite number of times $\tau_k$.

Consequently, there exists a finite time $T$ (possibly random) such that almost surely $\norm { \nabla \bar g (\theta_{t}) } < \kappa$ for $t \geq T$.  Since the original $\kappa>0$ was arbitrarily chosen, this shows that $\norm {\nabla \bar g (\theta_{t}) }  \rightarrow 0$ as  $t \rightarrow \infty$ almost surely.
\end{proof}

\section{Estimating the Coefficient Function of the Diffusion Term and Generalizations} \label{DiffusionCoefficientFunction}
We consider a diffusion $X_t \in \mathcal{X}=\mathbb{R}^{m}$:
\begin{eqnarray}
d X_t = f^{\ast}(X_t) dt + \sigma^{\ast}(X_t) d W_t.
\label{ClassofEqns2}
\end{eqnarray}
The goal is to statistically estimate a model $f(x, \theta)$ for $f^{\ast}(x)$ as well as a model $\sigma(x, \nu)\sigma^{\top}(x, \nu)$ for the diffusion coefficient $\sigma^{\ast}(x)\sigma^{\ast, \top}(x)$ where $\theta \in \mathbb{R}^n$ and $\nu \in \mathbb{R}^k$.  $W_t \in \mathbb{R}^m$ is a standard Brownian motion and $\sigma^{\ast}(\cdot) \in \mathbb{R}^{m \times m}$. The functions $f(x, \theta)$, $\sigma(x, \nu)$, $f^{\ast}(x)$, and $\sigma^{\ast}(x)$ may be non-convex.

The stochastic gradient descent update in continuous time follows the stochastic differential equations:
\begin{eqnarray}
d \theta_t &=&  \alpha_{t} \nabla_{\theta}  f(X_t, \theta_t) \big{ [}  d X_t -    f(X_t, \theta_t)  dt \big {]}, \notag \\
d \nu_t &=& \alpha_t \sum_{i,j}^m  \nabla_{\nu} \big{(} ( \sigma(X_t, \nu_t ) \sigma^{\top}(X_t, \nu_t ))_{i,j} \big{)}  \big { [}  d \la X_t, X_t \ra_{i,j}  - ( \sigma(X_t, \nu_t ) \sigma^{\top}(X_t, \nu_t ) )_{i,j} dt  \big{]} ,  \notag
\label{SDEMain2}
\end{eqnarray}
where $\la X_t, X_t \ra \in \mathbb{R}^{m \times m}$ is the quadratic variation matrix of $X$.  Since we observe the path of $X_t$, we also observe the path of the quadratic variation $\la X_t, X_t \ra$.

Let us set:
\begin{eqnarray}
g(x,\theta)   &=&  \frac{1}{2}\left\|f(x,\theta)-f^{\ast}(x)\right\|^{2}  \notag \\
w(x, \nu)   &=&  \frac{1}{2}\left\|    \sigma(x, \nu) \sigma^{\top}(x, \nu )   -   \sigma^{\ast}(x)   \sigma^{\ast, \top}(x)    \right\|^{2}.  \notag
\end{eqnarray}

We assume that $\sigma^{\ast}(x)$ is such that the process $X_{t}$ is ergodic with a unique invariant measure (for example one may assume that it is non-degenerate, i.e., bounded away from zero and bounded by above). In addition, we assume that $w(x,\nu)$ satisfies the same assumptions as $g(x,\theta)$ does in Condition \ref{A:Assumption1}.

From the previous results in Section \ref{S:ProofMainTheorem}, $\lim_{t\rightarrow\infty}\|\nabla \bar{g}(\theta_{t})\|=0$ as $t \rightarrow \infty$ with probability 1.  Let's study the convergence of the stochastic gradient descent algorithm (\ref{SDEMain2}) for $\nu_t$.  By It\^{o}'s formula,
\begin{eqnarray*}
\bar w(\nu_{\sigma}) - \bar w (\nu_{\tau}) &=&  -  \int_{\tau}^{\sigma} \alpha_{s} \left\| \nabla \bar w (\nu_s) \right\|^2 ds +\int_{\tau}^{\sigma} \alpha_{s} \left<\nabla_{\nu} \bar w(\nu_{s}), \nabla_{\nu} \bar{w}(\nu_{s})- \nabla_{\nu} w(X_{s},\nu_{s})\right>ds
\end{eqnarray*}
Applying exactly the same procedure as in Section \ref{S:ProofMainTheorem}, $\lim_{t\rightarrow\infty}\|\nabla \bar{w}(\nu_{t})\|=0$ as $t \rightarrow \infty$ with probability 1.  We omit the details as the proof is exactly the same as in Section \ref{S:ProofMainTheorem}.

Notice also that  $\sigma^{\ast}(x)$ is not identifiable; for example, $X_t$ has the same distribution under the diffusion coefficient $-\sigma^{\ast}(x)$.  Only $\sigma^{\ast}(x) \sigma^{\ast, \top}(x)$ is identifiable.  We are therefore essentially estimating a model $\sigma(x, \nu) \sigma^{\top}(x, \nu)$  for $\sigma^{\ast}(x) \sigma^{\ast, \top}(x)$.

We close this section with the following remark.
\begin{remark}
The proof of Theorem \ref{T:MainTheorem} makes it clear that if appropriate assumptions on $\nabla_{\theta}f$ and $\nabla_{g}\theta$ are made such that $\sup_{t>0}\mathbb{E}|\theta_{t}^{q}|<C$ for appropriate $0<q,C<\infty$, then one can relax Condition \ref{A:Assumption1} on $\nabla_{\theta}g$ to allow at least linear growth with respect to $\theta$.
\end{remark}

\section{Model Estimation: Numerical Analysis} \label{NumericalAnalysis1}
We implement SGDCT for several applications and numerically analyze the convergence.  Section \ref{OU} studies continuous-time stochastic gradient descent for the Ornstein-Uhlenbeck process, which is widely used in finance, physics, and biology.  Section \ref{OU2} studies the multidimensional Ornstein-Uhlenbeck process.  Section \ref{Burger} estimates the diffusion coefficient in Burger's equation with continuous-time stochastic gradient descent.  Burger's equation is a widely-used nonlinear partial differential equation which is important to fluid mechanics, acoustics, and aerodynamics.  Burger's equation is extensively used in engineering.  In Section \ref{RLexample}, we show how SGDCT can be used for reinforcement learning.  In the final example, the drift and volatility functions for the multidimensional CIR process are estimated.  The CIR process is widely used in financial modeling.

\subsection{Ornstein-Uhlenbeck process} \label{OU}
The Ornstein-Uhlenbeck (OU) process $X_t \in \mathbb{R}$ satisfies the stochastic differential equation:
\begin{eqnarray}
d X_t = c (m - X_t) dt + d W_t.
\label{OUeqn1d}
\end{eqnarray}
We use continuous-time stochastic gradient descent to learn the parameters $\theta = (c,m) \in \mathbb{R}^2$.

For the numerical experiments, we use an Euler scheme with a time step of $10^{-2}$.  The learning rate is $\alpha_t = \min(\alpha, \alpha / t )$ with $\alpha = 10^{-2}$.  We simulate data from (\ref{OUeqn1d}) for a particular $\theta^{\ast}$ and the stochastic gradient descent attempts to learn a parameter $\theta_t$ which fits the data well.  $\theta_t$ is the statistical estimate for $\theta^{\ast}$ at time $t$.  If the estimation is accurate, $\theta_t$ should of course be close to $\theta^{\ast}$.  This example can be placed in the form of the original class of equations (\ref{ClassofEqns}) by setting $f(x, \theta) = c (m - x)$ and $f^{\ast}(x) = f(x, \theta^{\ast})$.

We study $10,500$ cases.  For each case, a different $\theta^{\ast}$ is generated uniformly at random in the range $[1, 2] \times [1,2]$.  For each case, we solve for the parameter $\theta_t$ over the time period $[0,T]$ for $T = 10^6$.  To summarize:
\begin{itemize}
\item For cases n = 1 to 10,500
\begin{itemize}
\item Generate a random $\theta^{\ast}$ in $[1, 2] \times [1,2]$
\item Simulate a single path of $X_t$ given $\theta^{\ast}$ and simultaneously solve for the path of $\theta_t$ on $[0,T]$
\end{itemize}
\end{itemize}

The accuracy of $\theta_t$ at times $t = 10^2, 10^3, 10^4, 10^5$, and $10^6$ is reported in Table \ref{OUerror}.  Figures \ref{OUerrorPlot} and \ref{OUMSEerrorPlot1d} plot the mean error in percent and mean squared error (MSE) against time.   In the table and figures, the ``error" is $| \theta_t^n - \theta^{\ast,n}|$ where $n$ represents the $n$-th case.  The ``error in percent" is $100 \times \frac{ | \theta_t^n - \theta^{\ast,n}| }{ | \theta^{\ast,n} |}$.  The ``mean error in percent" is the average of these errors, i.e. $\frac{100}{N} \sum_{n=1}^N \frac{ | \theta_t^n - \theta^{\ast,n}| }{ | \theta^{\ast,n} |}$.

\begin{figure}[h!]
\centering
\includegraphics[scale=0.8]{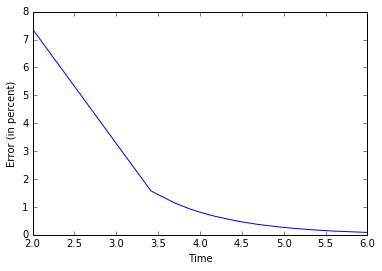}
\caption{Mean error in percent plotted against time.  Time is in log scale.}
\label{OUerrorPlot}
\end{figure}

\begin{figure}[h!]
\centering
\includegraphics[scale=0.8]{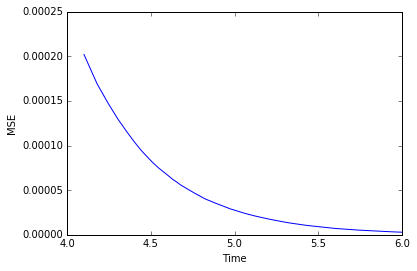}
\caption{Mean squared error plotted against time.  Time is in log scale.}
\label{OUMSEerrorPlot1d}
\end{figure}

\begin{table}[ht!]
\begin{center}
 \begin{tabular}{  |c| c  | c | c| c | c |}
   \hline
Error/Time &  $10^{2}$ & $10^3$ & $10^4$ & $10^5$ & $10^6$  \\ \hline \hline
Maximum Error & .604   & .2615 & .0936  & .0349  & .0105  \\ \hline
99\% quantile of error & .368  & .140 & .0480  & .0163  & .00542   \\ \hline
99.9\% quantile of error &  .470  & .1874 & .0670 & .0225   & .00772 \\ \hline
Mean squared error & $1.92 \times 10^{-2} $  &   $2.28 \times 10^{-3}$ & $2.52 \times 10^{-4}$ & $2.76 \times 10^{-5}$ & $2.90 \times 10^{-6}$   \\ \hline
Mean Error in percent & 7.37  & 2.497 & 0.811  & 0.264 & 0.085   \\ \hline
Maximum error in percent & 59.92 & 20.37 & 5.367   & 1.79  & 0.567    \\ \hline
99\% quantile of error in percent & 25.14  & 9.07 & 3.05  & 1.00 & 0.323  \\ \hline
99.9\% quantile of error in percent & 34.86   & 12.38   & 4.12  & 1.30   & 0.432   \\ \hline
 \end{tabular}
\end{center}
\caption{\label{OUerror}  Error at different times for the estimate $\theta_t$ of $\theta^{\ast}$ across $10,500$ cases.  The ``error" is $| \theta_t^n - \theta^{\ast,n}|$ where $n$ represents the $n$-th case.  The ``error in percent" is $100 \times \frac{ | \theta_t^n - \theta^{\ast,n}| }{ | \theta^{\ast,n} |}.$}
\end{table}

Finally, we also track the objective function $\bar g(\theta_t) $ over time.  Figure \ref{OUbarGplot} plots the error $\bar g(\theta_t) $ against time.  Since the limiting distribution $\pi(x) $ of (\ref{OUeqn}) is Gaussian with mean $m^{\ast}$ and variance $\frac{1}{2 c ^{\ast} }$, we have that:

\begin{eqnarray*}
\bar g(\theta) &=& \int  \bigg ( c^{\ast}( m^{\ast} - x) - c (m -x) \bigg ) ^2  \pi(x) dx  \notag \\
&=& (c^{\ast} m^{\ast} - c m )^2  + (c^{\ast} - c)^2 ( \frac{1}{2 c^{\ast}} + (m^{\ast})^2) + 2 (c^{\ast} m^{\ast} - c m ) (c-c^{\ast} ) m^{\ast}
\end{eqnarray*}

\begin{figure}[h!]
\centering
\includegraphics[scale=0.8]{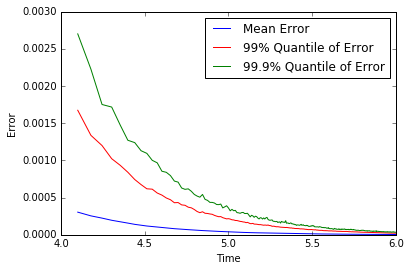}
\caption{The error $\bar g(\theta_t)$ plotted against time.  The mean error and the quantiles of the error are calculated from the 10,500 cases.  Time is in log scale.}
\label{OUbarGplot}
\end{figure}

\subsection{Multidimensional Ornstein-Uhlenbeck process} \label{OU2}
The multidimensional Ornstein-Uhlenbeck process $X_t \in \mathbb{R}^d$ satisfies the stochastic differential equation:
\begin{eqnarray}
d X_t =  (M - A X_t) dt + d W_t.
\label{OUeqn}
\end{eqnarray}
We use continuous-time stochastic gradient descent to learn the parameters $\theta = (M,A) \in \mathbb{R}^d \times \mathbb{R}^{d \times d}$.

For the numerical experiments, we use an Euler scheme with a time step of $10^{-2}$.  The learning rate is $\alpha_t = \min(\alpha, \alpha / t )$ with $\alpha = 10^{-1}$.  We simulate data from (\ref{OUeqn}) for a particular $\theta^{\ast}  = (M^{\ast}, A^{\ast})$ and the stochastic gradient descent attempts to learn a parameter $\theta_t$ which fits the data well.  $\theta_t$ is the statistical estimate for $\theta^{\ast}$ at time $t$.  If the estimation is accurate, $\theta_t$ should of course be close to $\theta^{\ast}$.  This example can be placed in the form of the original class of equations (\ref{ClassofEqns}) by setting $f(x, \theta) =  M -A x$ and $f^{\ast}(x) = f(x, \theta^{\ast})$.

The matrix $A^{\ast}$ must be generated carefully to ensure that $X_t$ is ergodic and has a stable equilibrium point.  If some of $A^{\ast}$'s eigenvalues have negative real parts, then $X_t$ can become unstable and grow arbitrarily large.  Therefore, we randomly generate matrices $A^{\ast}$ which are strictly diagonally dominant.  $A^{\ast}$'s eigenvalues are therefore guaranteed to have positive real parts and $X_t$ will be ergodic.  To generate random strictly diagonally dominant matrices $A^{\ast}$, we first generate $A_{i,j}^{\ast}$ uniformly at random in the range $[1,2]$ for $i \neq j$.  Then, we set $A_{i,i}^{\ast} = \sum_{j \neq i} A_{i,j}^{\ast} + U_{i,i}$ where $U_{i,i}$ is generated randomly in $[1,2]$.  $M_i^{\ast}$ for $i =1,\ldots, d$ is also generated randomly in $[1,2]$.

We study $525$ cases and analyze the error in Table \ref{OU2error}.  Figures \ref{OU2errorPlot} and \ref{OUMSEerrorPlot} plot the error over time.

\begin{table}[ht!]
\begin{center}
 \begin{tabular}{  |c| c  | c | c| c | c |}
   \hline
Error/Time &  $10^{2}$ & $10^3$ & $10^4$ & $10^5$ & $10^6$  \\ \hline \hline
Maximum Error & 2.89   & .559 & .151  & .043  & .013  \\ \hline
99\% quantile of error & 2.19  & .370 & .0957  & .0294  & .00911   \\ \hline
99.9\% quantile of error &  2.57  & .481 & .118 & .0377   & .0117 \\ \hline
Mean squared error & $8.05 \times 10^{-1} $  &   $2.09 \times 10^{-2}$ & $1.38 \times 10^{-3}$ & $1.29 \times 10^{-4}$ & $1.25 \times 10^{-5}$   \\ \hline
Mean Error in percent & 34.26  & 6.18 & 1.68  &  0.52 & 0.161   \\ \hline
Maximum error in percent & 186.3 & 41.68 &  10.98   & 3.81  & 1.03    \\ \hline
99\% quantile of error in percent & 109.2  & 23.9 & 6.98  & 2.15 & 0.657  \\ \hline
99.9\% quantile of error in percent & 141.2   & 31.24   & 8.64  & 2.84  & 0.879   \\ \hline
 \end{tabular}
\end{center}
\caption{\label{OU2error}  Error at different times for the estimate $\theta_t$ of $\theta^{\ast}$ across $525$ cases.  The ``error" is $| \theta_t^n - \theta^{\ast,n}|$ where $n$ represents the $n$-th case.  The ``error in percent" is $100 \times \frac{ | \theta_t^n - \theta^{\ast,n}| }{ | \theta^{\ast,n} |}.$}
\end{table}

\begin{figure}[h!]
\centering
\includegraphics[scale=0.8]{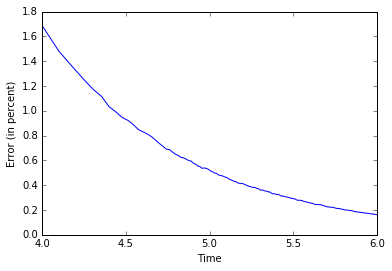}
\caption{Mean error in percent plotted against time.  Time is in log scale.}
\label{OU2errorPlot}
\end{figure}

\begin{figure}[h!]
\centering
\includegraphics[scale=0.8]{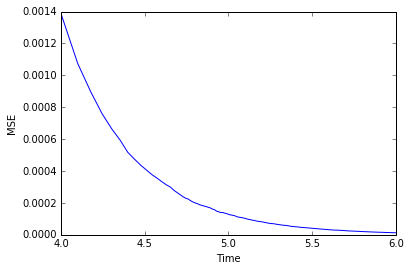}
\caption{Mean squared error plotted against time.  Time is in log scale.}
\label{OUMSEerrorPlot}
\end{figure}

\subsection{Burger's Equation} \label{Burger}
The stochastic Burger's equation that we consider is given by:
\begin{eqnarray}
\frac{\partial u}{\partial t}(t,x) = \theta \frac{\partial^2 u}{\partial x^2} - u(t,x) \frac{\partial u}{\partial x}(t,x) + \sigma \frac{ \partial^2 W(t,x)}{\partial t \partial x},
\label{BurgerStochastic}
\end{eqnarray}
where $x \in [0,1]$ and $W(t,x)$ is a Brownian sheet.  The finite-difference discretization of (\ref{BurgerStochastic}) satisfies a system of nonlinear stochastic differential equations (for instance, see \cite{DavieGaines} or \cite{Gyongy}).  We use continuous-time stochastic gradient descent to learn the diffusion parameter $\theta$.

We use the following finite difference scheme for Burger's equation:
\begin{eqnarray}
d u (t, x_i)  = \theta \frac{  u(t, x_{i+1} ) - 2 u(t, x_i) + u(t, x_{i-1}) }{ \Delta x^2} dt - u(t, x_i) \frac{ u(t, x_{i+1}) - u(t, x_{i-1}) }{ 2 \Delta x}dt +  \frac{\sigma}{ \sqrt{\Delta x} } d W_t^i,
\label{FiniteDifference0}
\end{eqnarray}

For our numerical experiment, the boundary conditions $u(t,x=0) = 0$ and $u(t, x = 1) = 1$ are used and $\sigma = 0.1$.   (\ref{FiniteDifference0}) is simulated with the Euler scheme (i.e., we solve Burger's equation with explicit finite difference).  A spatial discretization of $\Delta x = .01$ and a time step of $10^{-5}$ are used.  The learning rate is $\alpha_t = \min(\alpha, \alpha / t )$ with $\alpha = 10^{-3}$.  The small time step is needed to avoid instability in the explicit finite difference scheme.  We simulate data from (\ref{BurgerStochastic}) for a particular diffusion coefficient $\theta^{\ast}$ and the stochastic gradient descent attempts to learn a diffusion parameter $\theta_t$ which fits the data well.  $\theta_t$ is the statistical estimate for $\theta^{\ast}$ at time $t$.  If the estimation is accurate, $\theta_t$ should of course be close to $\theta^{\ast}$.

This example can be placed in the form of the original class of equations (\ref{ClassofEqns}).  Let $f_i$ be the $i$-th element of the function $f$.  Then, $f_i(u, \theta) = \theta \frac{  u(t, x_{i+1} ) - 2 u(t, x_i) + u(t, x_{i-1}) }{ \Delta x^2} - u(t, x_i) \frac{ u(t, x_{i+1}) - u(t, x_{i-1}) }{ 2 \Delta x} $.  Similarly, let $f_i^{\ast}$ be the $i$-th element of the function $f^{\ast}$.  Then, $f_i^{\ast}(u) = f_i(u, \theta^{\ast})$.

We study $525$ cases.  For each case, a different $\theta^{\ast}$ is generated uniformly at random in the range $[.1, 10]$.  This represents a wide range of physical cases of interest, with $\theta^{\ast}$ ranging over two orders of magnitude.  For each case, we solve for the parameter $\theta_t$ over the time period $[0,T]$ for $T = 100$.

The accuracy of $\theta_t$ at times $t = 10^{-1}, 10^0, 10^1,$ and $10^2$ is reported in Table \ref{Burgererror}.  Figures \ref{BurgererrorPlot} and \ref{BurgerMSEerrorPlot} plot the mean error in percent and mean squared error against time.  The convergence of $\theta_t$ to $\theta^{\ast}$ is fairly rapid in time.

\begin{table}[ht!]
\begin{center}
 \begin{tabular}{  |c| c  | c | c| c  |}
   \hline
Error/Time &  $ 10^{-1} $ & $ 10^0 $ & $10^1$ & $10^2$   \\ \hline \hline
Maximum Error & .1047   & .106 & .033  & .0107  \\ \hline
99\% quantile of error & .08  & .078 & .0255  & .00835   \\ \hline
Mean squared error & $1.00 \times 10^{-3} $  &   $9.25 \times 10^{-4}$ &  $1.02 \times 10^{-4}$ & $1.12 \times 10^{-5}$   \\ \hline
Mean Error in percent & 1.26 & 1.17 &  0.4  & 0.13   \\ \hline
Maximum error in percent & 37.1 & 37.5 & 9.82     & 4.73    \\ \hline
99\% quantile of error in percent & 12.6  & 18.0  & 5.64   & 1.38  \\ \hline
 \end{tabular}
\end{center}
\caption{\label{Burgererror}  Error at different times for the estimate $\theta_t$ of $\theta^{\ast}$ across $525$ cases.  The ``error" is $| \theta_t^n - \theta^{\ast,n}|$ where $n$ represents the $n$-th case.  The ``error in percent" is $100 \times  | \theta_t^n - \theta^{\ast,n}| / | \theta^{\ast,n} |.$ }
\end{table}

\begin{figure}[h!]
\centering
\includegraphics[scale=0.8]{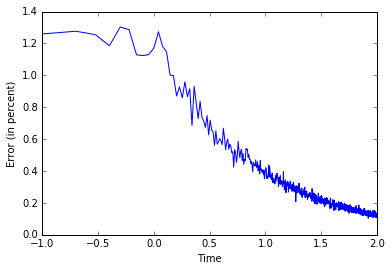}
\caption{Mean error in percent plotted against time.  Time is in log scale.}
\label{BurgererrorPlot}
\end{figure}

\begin{figure}[h!]
\centering
\includegraphics[scale=0.8]{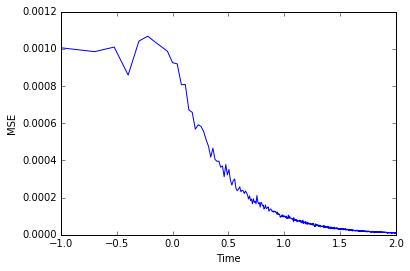}
\caption{Mean squared error plotted against time.  Time is in log scale.}
\label{BurgerMSEerrorPlot}
\end{figure}

\subsection{Reinforcement Learning} \label{RLexample}
We consider the classic reinforcement learning problem of balancing a pole on a moving cart (see \cite{Barto83}).  The goal is to balance a pole on a cart and to keep the cart from moving outside the boundaries via applying a force of $\pm 10$ Newtons.

The position $x$ of the cart, the velocity $\dot x$ of the cart, angle of the pole $\beta$, and angular velocity $\dot \beta$ of the pole are observed.  The dynamics of $s = (x, \dot x, \beta, \dot \beta)$ satisfy a set of ODEs (see \cite{Barto83}):
\begin{eqnarray}
\ddot \beta_t &=& \frac{ g \sin \beta_t + \cos \beta_t [  \frac{-F_t - ml \dot \beta_t^2 \sin \beta_t + \mu_c \textrm{sgn}( \dot x_t)}{m_c + m}   ] - \frac{\mu_p \dot \beta_t}{m l}  }{l [ \frac{4}{3} - \frac{m}{m_c} \frac{ \cos^2 \beta_t}{m_c + m} ] }, \notag \\
\ddot x_t &=&  \frac{F_t + ml [ \dot \beta_t^2 \sin \beta_t - \ddot \beta_t \cos \beta_t ]  - \mu_c \textrm{sgn}( \dot x_t) }{  m_c + m},
\label{CartpoleDynamicsODEs}
\end{eqnarray}
where $g$ is the acceleration due to gravity, $m_c$ is the mass of the cart, $m$ is the mass of the pole, $2l$ is the length of the pole,
$\mu_c$ is the coefficient of friction of the cart on the ground, $\mu_p$ is the coefficient of friction of the pole on the cart, and $F_t \in \{-10, 10 \}$ is the force
applied to the cart.

For this example, $f^{\ast}(s) = (\dot x, \ddot x, \dot \beta, \ddot \beta)$.  The model $f(s, \theta) = ( f_1(s, \theta), f_2(s, \theta), f_3(s, \theta), f_4(s, \theta) )$ where $f_i(s, \theta)$ is a single-layer neural network with rectified linear units.

\begin{eqnarray}
f_i(s, \theta) = W^{2,i} h(W^{1,i} s + b^{1,i}) +b^{2,i},
\end{eqnarray}
where $\theta = \{ W^{2,i}, W^{1,i}, b^{1,i}, b^{2,i} \}_{i=1}^4$ and $h(z) = (\sigma(z_1), \ldots, \sigma(z_d))$ for $z \in \mathbb{R}^d$.  The function $\sigma: \mathbb{R} \rightarrow \mathbb{R}$ is a rectified linear unit (ReLU): $\sigma(v) = \max ( v, 0)$.  We learn the parameter $\theta$ using continuous-time stochastic gradient descent.

The boundary is $x = \pm 2.4$ meters and the pole must not be allowed to fall below $\beta = \frac{24}{360\pi}  $ radians (the frame of reference is chosen such that the perfectly upright is $0$ radians).  A reward of $+1$ is received every $0.02$ seconds if $\norm{x} \leq 2.4$ and $\norm{\theta} \leq  \frac{24}{360\pi} $.  A reward of $-100$ is received (and the episode ends) if the cart moves beyond $x = \pm 2.4$ or the pole falls below $\beta = \frac{24}{360\pi}$ radians.   The sum of these rewards across the entire episode is the reward for that episode.  The initial state $(x, \dot x, \beta, \dot \beta)$ at the start of an episode is generated uniformly at random in $[-.05, .05]^4$.  For our numerical experiment, we assume that the rule for receiving the rewards and the distribution of the initial state are both known.  An action of $\pm 10$ Newtons may be chosen every $0.02$ seconds.  This force is then applied for the duration of the next $0.02$ seconds.  The system (\ref{CartpoleDynamicsODEs}) is simulated using an Euler scheme with a time step size of $10^{-3}$ seconds.

The goal, of course, is to statistically learn the optimal actions in order to achieve the highest possible reward.  This requires both: 1) statistically learning the physical dynamics of $(x, \dot x, \beta, \dot \beta)$ and 2) finding the optimal actions given these dynamics in order to achieve the highest possible reward.  The dynamics $(x, \dot x, \beta, \dot \beta)$ satisfy the set of ODEs (\ref{CartpoleDynamicsODEs}); these dynamics can be learned using continuous-time stochastic gradient descent.  We use a neural network for $f$.  Given the estimated dynamics $f$, we use a policy gradient method to estimate the optimal actions.  The approach is summarized below.

\begin{itemize}
\item For episodes $0,1, 2, \ldots$:
\begin{itemize}
\item For time $[0, T_{\textrm{end of episode}}]$:
\begin{itemize}
\item Update the model $f(s, \theta)$ for the dynamics using continuous-time stochastic gradient descent.
\end{itemize}
\item Periodically update the optimal policy $\mu(s,a, \theta^{\mu})$ using policy gradient method.  The optimal policy is learned using data simulated from the model $f(s, \theta)$.  Actions are randomly selected via the policy $\mu$.
\end{itemize}
\end{itemize}

The policy $\mu$ is a neural network with parameters $\theta^{\mu}$.  We use a single hidden layer with rectified linear units followed by a softmax layer for $\mu(s,a, \theta^{\mu})$ and train it using policy gradients.\footnote{Let $r_{e,t}$ be the reward for episode $e$ at time $t$.  Let $R_{t,e} = \sum_{t' = t+1}^{T_{\textrm{end of episode}} } \gamma^{t'-t} r_{e,t'}$ be the cumulative discounted reward from episode $e$ after time $t$ where $\gamma \in [0,1]$ is the discount factor.  Stochastic gradient descent is used to learn the parameter $\theta^{\mu}$: $\theta^{\mu} \leftarrow \theta^{\mu} + \eta_{e}R_{t,e} \frac{\partial}{\partial \theta^{\mu}} \log \mu(s_t, a_t, \theta^{\mu})$ where $\eta_e$ is the learning rate.   In practice, the cumulative discounted rewards are often normalized across an episode.}  The policy $\mu(s,a, \theta^{\mu})$ gives the probability of taking action $a$ conditional on being in the state $s$.
\begin{eqnarray}
\mathbb{P}[ F_t = 10 | s_t = s ] = \mu(s_t, 10, \theta^{\mu}) =  \sigma_0(W^2 h(W^1 s + b^1) + b^2 ),
\end{eqnarray}
where $\sigma_0(v) = \frac{e^v}{1 + e^v}$.  Of course, $ \mathbb{P}[ F_t = -10 | s_t = s ] = \mu(s, -10, \theta^{\mu}) = 1-\mu(s, 10, \theta^{\mu}) $.

$525$ cases are run, each for $25$ hours.  The optimal policy is learned using the estimated dynamics $f(s, \theta)$ and is updated every $5$ episodes.  Table \ref{CartpoleReward} reports the results at fixed episodes using continuous-time stochastic gradient descent.  Table \ref{CartpoleMinEpisode} reports statistics on the number of episodes required until a target episodic reward ($100$, $500$, $1000$) is first achieved.

\begin{table}[ht!]
\begin{center}
 \begin{tabular}{  |c| c  | c | c| c  | c|}
   \hline
Reward/Episode &  $ 10 $ & $ 20 $ & $30$ & $40$  & $45$   \\ \hline \hline
Maximum Reward & -20   & 981 & $2.21 \times 10^4$ & $ 6.64 \times 10^5 $ & $9.22 \times 10^5$ \\ \hline
90\% quantile of reward & -63  & 184 & 760  & 8354   & $1.5 \times 10^4$ \\ \hline
Mean reward & -78  &   67 &  401 & 5659 & $1.22 \times 10^4$   \\ \hline
10\% quantile of reward & -89 & -34 &  36  & 69  & 93   \\ \hline
Minimum reward & -92 & -82  & -61   & -46 & -23 \\ \hline
 \end{tabular}
\end{center}
\caption{\label{CartpoleReward}  Reward at the $k$-th episode across the $525$ cases using continuous-time stochastic gradient descent to learn the model dynamics.}
\end{table}

\begin{table}[ht!]
\begin{center}
 \begin{tabular}{  |c| c  | c | c| }
   \hline
Number of episodes/Target reward &  $ 100 $ & $ 500 $ & $1000$   \\ \hline \hline
Maximum & 39   & 134 & 428 \\ \hline
90\% quantile & 23  & 49 & 61 \\ \hline
Mean & 18 &   34 &  43   \\ \hline
10\% quantile & 13 & 21 &  26    \\ \hline
Minimum & 11 & 14  & 17    \\ \hline
 \end{tabular}
\end{center}
\caption{\label{CartpoleMinEpisode}  For each case, we record the number of episodes required until the target reward is first achieved using continuous-time stochastic gradient descent.  Statistics (maximum, quantiles, mean, minimum) for the number of episodes required until the target reward is first achieved.}
\end{table}

Alternatively, one could directly apply policy gradient to learn the optimal action using the observed data.  This approach does not use continuous-time stochastic gradient descent to learn the model dynamics, but instead directly learns the optimal policy from the data.  Again using 525 cases, we report the results in Table \ref{CartpoleRewardPG} for directly learning the optimal policy without using continuous-time stochastic gradient descent to learn the model dynamics.  Comparing Tables \ref{CartpoleReward} and \ref{CartpoleRewardPG}, it is clear that using continuous-time stochastic gradient descent to learn the model dynamics allows for the optimal policy to be learned significantly more quickly.  The rewards are much higher when using continuous-time stochastic gradient descent (see Table \ref{CartpoleReward}) than when not using it (see Table \ref{CartpoleRewardPG}).

\begin{table}[ht!]
\begin{center}
 \begin{tabular}{  |c| c  | c | c| c  | c|c| c|}
   \hline
Reward/Episode &  $ 10 $ & $ 20 $ & $30$ & $40$  & $100$  &  500   &  750    \\ \hline \hline
Maximum Reward & 51   &  1 & 15 & 77 & 121 &  1748   &  $1.91 \times 10^5$  \\ \hline
90\% quantile of reward & -52  & -48 & -42  & 8354   & -11   &  345   &  2314  \\ \hline
Mean reward & -73  &   -72 &  -69 & -68 & -53    &  150  &  1476 \\ \hline
10\% quantile of reward & -88 & -88 &  -87 & 69  & -83    & -1     &  63  \\ \hline
Minimum reward & -92 & -92  & -92   & -92 & -92   &    -81 & -74   \\ \hline
 \end{tabular}
\end{center}
\caption{\label{CartpoleRewardPG}  Reward at the $k$-th episode across the $525$ cases using policy gradient to learn the optimal policy.}
\end{table}

\subsection{Estimating both the drift and volatility functions for the multidimensional CIR process}  \label{CIR}
We now implement an example where SGDCT is used to estimate both the drift function and the volatility function.  The multidimensional CIR process $X_t \in \mathbb{R}^d$ is:
\begin{eqnarray}
d X_t = c( m- X_t ) dt + \sqrt{X_t} \odot \sigma d W_t,
\label{CIR1}
\end{eqnarray}
where $\odot$ is element-wise multiplication, $m \in \mathbb{R}^d$, $c,\sigma \in \mathbb{R}^{d \times d}$, $W_t \in \mathbb{R}^d$, with $c$ being a positive definite matrix.  The CIR process is often used
for modeling interest rates.

In equation (\ref{CIR1}), $f(x, \theta) = c (m -x )$ where $\theta = (c, m)$.  $f^{\ast}(x) = f(x, \theta^{\ast})$ where $\theta^{\ast} = (c^{\ast}, m^{\ast})$.  The volatility model is $\sigma(x, \nu) = \sqrt{x} \odot \nu$ and $\sigma^{\ast}(x) = \sigma(x, \nu^{\ast})$ where $\nu, \nu^{\ast} \in \mathbb{R}^{d \times d}$.  Table \ref{CIRerror} reports the accuracy of SGDCT for estimating the drift and volatility functions of the CIR process.

\begin{table}[ht!]
\begin{center}
 \begin{tabular}{  |c| c  | c | c| }
   \hline
Error/Parameter &  $c$ & $m$ &  $ ( \sqrt{X_t} \odot \sigma)^{\top} ( \sqrt{X_t} \odot \sigma) $  \\ \hline \hline
Maximum Error &  0.0157   & 0.009  & 0.010   \\ \hline
99\% quantile of error & 0.010  & 0.007 & 0.008    \\ \hline
99.9\% quantile of error &  0.0146  & 0.009 & 0.010  \\ \hline
Mean squared error & $1.49 \times 10^{-5} $  &   $6.65 \times 10^{-6}$ & $4.21 \times 10^{-6}$   \\ \hline
Mean Error in percent & 0.21  & 0.137 &  0.0623   \\ \hline
Maximum error in percent & 1.12 & 0.695 & 0.456     \\ \hline
99\% quantile of error in percent & 0.782  & 0.506 & 0.415  \\ \hline
99.9\% quantile of error in percent & 1.06   & 0.616   & 0.455   \\ \hline
 \end{tabular}
\end{center}
\caption{\label{CIRerror}  Accuracy is reported in percent and averaged across $317$ simulations.  Each simulation has a different random initialization for $c,m,$ and $\sigma$. The dimension $d=3$, the time step size is $10^{-2}$, and accuracy is evaluated at the final time $5 \times 10^5$.  $X_t$ is simulated using \cite{Alfonsi}.  Observations of the quadratic variation are generated from $( \sqrt{X_t} \odot \sigma)^{\top} ( \sqrt{X_t} \odot \sigma)$ at times $t = 0, .01, .02, \ldots$.  $ ( \sqrt{X_t} \odot \sigma)^{\top} ( \sqrt{X_t} \odot \sigma) $  is the quadratic variation per unit of time.  For each simulation, the average error (or average percent error) for the quadratic variation per unit time is calculated by averaging across many points in the path of $X_t$.  Then, the statistics in the third column of the table are calculated using the average errors (or average percent errors) from the $317$ simulations.}
\end{table}

\section{American Options} \label{NumericalAnalysis2}
High-dimensional American options are extremely computationally challenging to solve with traditional numerical methods such as finite difference.  Here we propose a new approach using statistical learning to solve high-dimensional American options.  SGDCT achieves a high accuracy on two benchmark problems with 100 dimensions.

\subsection{Q-learning}
Before describing the SGDCT algorithm for American options, it is important to note that traditional stochastic gradient descent faces certain difficulties in this class of problems.  Some brief remarks are provided below regarding this fact; the authors plan to elaborate on these issues in more detail in a future work.   The well-known Q-learning algorithm uses stochastic gradient descent to minimize an approximation to the discrete-time Hamilton-Jacobi-Bellman equation.  To demonstrate the challenges and the issues that arise, consider using Q-learning to estimate the value function:
\begin{eqnarray}
%d X_t  &=& d W_t. \notag \\
V(x) &=& \mathbb{E} \bigg{[}  \int_0^{\infty} e^{- \gamma t} r(X_t) dt \bigg{|} X_0 = x \bigg{]},\quad  X_t  = x+  W_t,
\label{Qlearning1}
\end{eqnarray}
where $\gamma>0$ is a discount factor and $r(x)$ is a reward function. The function $Q(x, \theta)$ is an approximation for the value function $V(x)$.  The parameter $\theta$ must be estimated.  The traditional approach would discretize the dynamics (\ref{Qlearning1}) and then apply a stochastic gradient descent update to the objective function:
\begin{eqnarray}
\mathbb{E} \bigg{[} \bigg{(} r(X_t) \Delta +  \mathbb{E}[ e^{-\gamma \Delta} Q(X_{t+ \Delta}; \theta)  | X_t ] - Q(X_t \theta)  \bigg{)}^2 \bigg{]}.
\label{Qlearningobjective}
\end{eqnarray}
This results in the stochastic gradient descent algorithm:
\begin{eqnarray}
\theta_{t+ \Delta} &=& \theta_t  - \frac{ \alpha_t }{\Delta} \bigg{(} e^{- \gamma \Delta} \mathbb{E} \big{[} Q_{\theta} (X_{t+ \Delta}; \theta_t) \big{|} X_t \big{]} - Q_{\theta} (X_t; \theta_t) \bigg{)} \notag \\
&\times& \bigg{(}  r(X_t) \Delta +   e^{- \gamma \Delta} \mathbb{E} \big{[} Q(X_{t+ \Delta}; \theta_t) \big{|} X_t \big{]} - Q(X_t; \theta_t) \bigg{)}.
\label{Qlearning2}
\end{eqnarray}
Note that we have scaled the learning rate in (\ref{Qlearning2}) by $\frac{1}{\Delta}$.  This is the correct scaling for taking the limit $\Delta \rightarrow 0$.
The algorithm (\ref{Qlearning2}) has a major computational issue.  If the process $X_t$ is high-dimensional, $\mathbb{E} \big{[} Q(X_{t+ \Delta}; \theta_t) \big{|} X_t \big{]}$ is computationally challenging to calculate, and this calculation must be repeated for a large number of samples (millions to hundreds of millions).  It is also important to note that for the American option example that follows the underlying dynamics are known.  However, in reinforcement learning applications the transition probability is unknown, in which case $\mathbb{E} \big{[} Q(X_{t+ \Delta}; \theta_t) \big{|} X_t \big{]}$ cannot be calculated.  To circumvent these obstacles, the Q-learning algorithm ignores the inner expectation in (\ref{Qlearningobjective}), leading to the algorithm:
\begin{eqnarray}
\theta_{t+ \Delta} &=& \theta_t  - \frac{ \alpha_t }{\Delta} \bigg{(} e^{- \gamma \Delta} Q_{\theta} (X_{t+ \Delta}; \theta_t)  - Q_{\theta} (X_t; \theta_t) \bigg{)}  \bigg{(}  r(X_t) \Delta +   e^{- \gamma \Delta} Q(X_{t+ \Delta}; \theta_t)  - Q(X_t; \theta_t) \bigg{)}.
\label{Qlearning3}
\end{eqnarray}
Although now computationally efficient, the Q-learning algorithm (\ref{Qlearning3}) is now biased (due to ignoring the inner expectations).  Furthermore, when $\Delta \rightarrow 0$, the Q-learning algorithm (\ref{Qlearning3}) blows up.  A quick investigation shows that the term $\frac{1}{\Delta} (W_{t+\Delta} - W_t )^2 = O(1)$ arises while all other terms are $O(\Delta)$ or $O(\sqrt{\Delta})$.

The SGDCT algorithm is unbiased and computationally efficient.  It can be directly derived by letting $\Delta \rightarrow 0$ and using It\^{o}'s formula in (\ref{Qlearning2}):
\begin{eqnarray}
d \theta_t &=& - \alpha_t \bigg{(} \frac{1}{2} Q_{\theta xx}(X_t; \theta_t) - \gamma Q_{\theta}(X_t; \theta_t) \bigg{)}  \bigg{(} r(X_t) + \frac{1}{2} Q_{xx}(X_t; \theta_t) - \gamma Q(X_t; \theta_t)  \bigg{)}  dt.
\label{ContinuousTimeQlearning}
\end{eqnarray}
Note that computationally challenging terms in (\ref{Qlearning2}) become differential operators in (\ref{ContinuousTimeQlearning}), which are usually easier to evaluate.  This is one of the advantages of developing the theory in continuous time for continuous-time models.  Once the continuous-time algorithm is derived, it can be appropriately discretized for numerical solution.

\subsection{SGDCT for American Options}
Let $X_t \in \mathbb{R}^d$ be the prices of $d$ stocks.  The maturity date is time $T$ and the payoff function is $g(x): \mathbb{R}^d \rightarrow \mathbb{R}$.  The stock dynamics and value function are:
\begin{eqnarray}
d X_t^i &=&  \mu(X_t^i) dt +\sigma(X_t^i) d W_t^i, \notag \\
V(t,x) &=&  \sup_{\tau \geq t} \mathbb{E}[ e^{- r( \tau \wedge T)} g(X_{\tau \wedge T}) | X_t = x],
\end{eqnarray}
where $W_t \in \mathbb{R}^d$ is a Brownian motion.  The distribution of $W_t$ is specified by$\textrm{Var}[ W_t^i ] = t$ and $\textrm{Corr}[ W_t^i, W_t^j] = \rho_{i,j}$ for $i \neq j$.  The SGDCT algorithm for an American option is:
\begin{eqnarray}
 \theta_{\tau \wedge T}^{n+1} &=& \theta_0^n - \int_0^{\tau \wedge T} \alpha_t^{n+1} \bigg{(} \frac{\partial}{\partial t} Q_{\theta}(t, X_t; \theta_t^{n+1}) + \mathcal{L}_{x} Q_{\theta}(t, X_t; \theta_t^{n+1}) -r Q_{\theta} (t, X_t; \theta_t^{n+1}) \bigg{)}  \notag \\
 &\times& \bigg{(}\frac{\partial Q}{\partial t} (t, X_t; \theta_t^{n+1}) + \mathcal{L}_{x}Q (t, X_t; \theta_t^{n+1})  -r Q (t, X_t; \theta_t^{n+1})  \bigg{)}  dt \notag \\
 &+&  \alpha_{\tau \wedge T}^{n+1} Q_{\theta} (\tau \wedge T, X_{\tau \wedge T}; \theta_{\tau \wedge T}^{n+1}) \bigg{(} g(X_{\tau \wedge T}) - Q(\tau \wedge T, X_{\tau \wedge T}; \theta_{\tau \wedge T}^{n+1}) \bigg{)}, \notag \\
 \tau &=& \inf \{ t \geq 0 : Q(t, X_t; \theta_t^{n+1}) < g(X_t)  \}, \notag \\
 X_0 &\sim& \nu(dx).
\label{ContinuousTimeQlearning}
\end{eqnarray}
$\mathcal{L}_{x}$ is the infinitesimal generator for the $X$ process.  The continuous-time algorithm (\ref{ContinuousTimeQlearning}) is run for many iterations $n = 0,1,2, \ldots$ until convergence.  See the authors' paper \cite{DGM} for implementation details on pricing American options with deep learning.  %Formal derivation of algorithms such as  (\ref{ContinuousTimeQlearning}) will be discussed \cite{JK2}.

We implement the SGDCT algorithm (\ref{ContinuousTimeQlearning}) using a deep neural network for the function $Q(t, x; \theta)$.  Two benchmark problems are considered where semi-analytic solutions are available.  The SGDCT algorithm's accuracy is evaluated for American options in $d=100$ dimensions, and the results are presented in Table \ref{AmericanOptionTable}.

\begin{table}[ht!]
\begin{center}
 \begin{tabular}{  |c| c  | c | c| }
   \hline
Model &   Number of dimensions  &  Payoff function & Accuracy   \\ \hline \hline
Bachelier &   100                     &  $ \displaystyle g(x)=\max \big{(} \frac{1}{d} \sum_{i=1}^d x_i - K, 0 \big{)} $ & 0.1\%   \\ \hline
Black-Scholes & 100    & $\displaystyle g(x)=\max \big{(} (\prod_{i=1}^d x_i )^{1/d}- K, 0 \big{)} $              & 0.2\%    \\ \hline
 \end{tabular}
\end{center}
\caption{\label{AmericanOptionTable} For the Bachelier model, $\mu(x) = r -c$ and $\sigma(x) = \sigma$.  For Black-Scholes, $\mu(x) = (r -c) x$ and $\sigma(x) = \sigma x$.  All stocks are identical with correlation $\rho_{i,j} = .75$, volatility $\sigma = .25$, initial stock price $X_0 = 1$, dividend rate $c = 0.02$, and interest rate $r = 0$.  The maturity of the option is $T =2$ and the strike price is $K=1$.  The accuracy is reported for the price of the at-the-money American call option.}
\end{table}

\appendix
\section{On a related Poisson equation}\label{S:RegularityResults}

We recall the following regularity result from \cite{pardoux2003poisson} on the Poisson equations in the whole space, appropriately stated to cover our case of interest.
\begin{theorem} \label{T:RegularityPoisson}
Let Conditions \ref{A:LyapunovCondition} and \ref{A:Assumption1} be satisfied. Assume that $G(x,\theta)\in C^{\alpha,2}\left(\mathcal{X},\mathbb{R}^{n}\right)$,
\begin{equation}
\int_{\mathcal{X}}G(x,\theta)\pi(dx)=0,\label{Eq:CenteringCondition}
\end{equation}
and that for some positive constants $K$ and $q$,
\[
    \sum_{i=0}^{2}\left|\frac{\partial^{i} G}{\partial \theta^{i}}(x,\theta)\right|\leq K\left(1+|x|^{q}\right)
\]
Let $\mathcal{L}_{x}$ be the infinitesimal generator for the $X$ process. Then the Poisson equation
\begin{eqnarray}
& &\mathcal{L}_{x}u(x,\theta)=G(x,\theta),\quad\int_{\mathcal{X}}%
u(x,\theta)\pi(dx)=0 \label{Eq:CellProblem}
\end{eqnarray}
has a unique solution that satisfies $u(x,\cdot)\in C^{2}$ for every $x\in\mathcal{X}$, $\partial_{\theta}^{2}u\in C\left(\mathcal{X}\times\mathbb{R}^{n}\right)$ and there exist positive constants $K'$ and $q'$ such that
\[
    \sum_{i=0}^{2}\left|\frac{\partial^{i} u}{\partial \theta^{i}}(x,\theta)\right|+\left|\frac{\partial^{2} u}{\partial x\partial \theta}(x,\theta)\right|\leq K'\left(1+|x|^{q'}\right).
\]
\end{theorem}

\end{document}